\documentclass[11pt,a4paper]{article}
\usepackage{amsmath}
\usepackage{amssymb}
\usepackage{amsthm}
\usepackage{amsfonts}
\usepackage{latexsym}
\usepackage{verbatim} 
\usepackage{xcolor}
\usepackage[a4paper,width=16cm,top=2cm,bottom=2cm,footskip=1cm
]{geometry}

\usepackage[flushmargin]{footmisc}
\usepackage{color}                    
\usepackage{esvect}
\usepackage{mathrsfs}

\theoremstyle{plain}
\newtheorem{thm}{Theorem}[section]
\newtheorem{cor}{Corollary}[section]
\newtheorem{lem}{Lemma}[section]

\newtheorem{ass}{Assumption}[section]
\theoremstyle{remark}

\theoremstyle{definition}
\newtheorem{defn}{Definition}[section]

\newtheorem{rem}{Remark}[section]

\newcommand{\Complex}{\mathbb C}
\newcommand{\Real}{\mathbb R}
\newcommand{\N}{\mathbb N}
\newcommand{\ddbar}{\overline\partial}
\newcommand{\pr}{\partial}
\newcommand{\ol}{\overline}
\newcommand{\Td}{\widetilde}
\newcommand{\norm}[1]{\left\Vert#1\right\Vert}
\newcommand{\abs}[1]{\left\vert#1\right\vert}
\newcommand{\set}[1]{\left\{#1\right\}}
\newcommand{\To}{\rightarrow}

\title{Functional calculus and quantization commutes with reduction for Toeplitz operators on CR manifolds}
\author{Andrea Galasso and Chin-Yu Hsiao\footnote{\noindent{\bf Address:} Room 407, Chee-Chun Leung Cosmology Hall, National Taiwan University; {\bf ORCID iD:} 0000-0002-5792-1674; {\bf e-mail}: andrea.galasso@ncts.ntu.edu.tw {\bf Address:} Institute of Mathematics, Academia Sinica, 6F, Astronomy-Mathematics Building, No.1, Sec.4, Roosevel: {\bf ORCHID iD:} 0000-0002-1781-0013; {\bf email}: chsiao@math.sinica.edu.tw; chinyu.hsiao@gmail.com}}
\date{}

\begin{document}
\maketitle

\begin{abstract} Given a CR manifold with non-degenerate Levi form, we show that the operators of the functional calculus for Toeplitz operators are complex Fourier integral operators of Szeg\H{o} type. As an application, we establish semi-classical spectral dimensions for Toeplitz operators. We then consider  a CR manifold with a compact Lie group action $G$ and we establish quantization commutes with reduction for Toeplitz operators. Moreover, we also compute semi-classical spectral dimensions for $G$-invariant Toeplitz operators.
\end{abstract}
\tableofcontents


\textbf{Data Availability Statements:} Data sharing not applicable to this article as no datasets were generated or analysed during the current study.

\section{Introduction} 

The study of Toeplitz operators plays an important role in modern mathematical research, theoretical physics and it is closely related to various domain of mathematics such as operator theory, spectral theory, analysis, geometry and geometric quantization theory. Boutet de Monvel and Guillemin~\cite{bg} have been set up as an extension of the classical Toeplitz operators on Hardy spaces. They consider $\Pi_{\Sigma}P\Pi_{\Sigma}$, where $\Pi_{\Sigma}$ is the Szeg\H{o} projection associated to a given symplectic conic subset $\Sigma$ and $P$ is a pseudodifferential operator. The great interest for us is that the Toeplitz operators enjoy similar properties as pseudodifferential operators (symbolic caculus, Weyl law,...). In this work, we will study the functional calculus for Toeplitz operators on a CR manifold with non-degenerate Levi form which is fundamental for our study of Toeplitz operators.  We show that the operators of the functional calculus for Toeplitz operators are complex Fourier integral operators of Szeg\H{o} type. As an application, we establish semi-classical spectral dimensions for Toeplitz operators. We then consider  a CR manifold with a compact Lie group action $G$ and we establish quantization commutes with reduction for Toeplitz operators. Moreover, we also compute semi-classical spectral dimensions for $G$-invariant Toeplitz operators. 

Let us recall some prior literature about micro-local theory of Toeplitz operators, without the pretense of being complete. Toeplitz operators on strictly pseudoconvex domain were introduced in the monograph~\cite{bg}, in the survey \cite{p} it is  reviewed some recent results concerning eigenfunction asymptotics in this context, largely based on the generalization to the almost complex symplectic category by Shiffman and Zelditch. In the presence of symmetries, in \cite{mz} some Toeplitz operator type properties in semiclassical analysis was investigated based on the analytic localization techniques developed by Bismut and Lebeau. We refer to \cite{mm} for more comprehensive bibliographic notes, our approach is based on the Boutet de Monvel–Sj\"ostrand parametrix for the Szeg\H{o} kernel.

We now formulate the main results. They are based on \cite{hsiao}, \cite{hsiaohuang} and \cite{gh} where it is studied Toeplitz operators for CR manifolds with non-degenerate Levi-form, in particular in \cite{gh} we studied the properties of the corresponding star-product inspired by previous results as in \cite{bms} and \cite{schli}. We refer the reader to Section~\ref{s:prelim} for some notations and terminology used here. 

\subsection{Main results: Functional calculus for Toeplitz operators} \label{s-gue211217yyd}

Let $(X, T^{1,0}X)$ be a compact, orientable CR manifold of dimension $2n+1$, $n\geq 1$, where $T^{1,0}X$ is a CR structure of $X$. Since $X$ is orientable, there always exists a real non-vanishing $1$-form $\omega_0\in\mathcal{C}^{\infty}(X,T^*X)$ so that $\langle\,\omega_0(x)\,,\,u\,\rangle=0$, for every $u\in T^{1,0}_xX\oplus T^{0,1}_xX$, for every $x\in X$. We fix $\omega_0$ and let $\mathcal{L}_x$ be the Levi form at $x\in X$  with respect to $\omega_0$ (see \eqref{e-gue211216ycdp}). In this work, we assume that the Levi form is non-degenerate. The Reeb vector field $R\in\mathcal{C}^\infty(X,TX)$ is defined to be the non-vanishing vector field determined by 
\[	\omega_0(R)=-1\quad \text{ and }\quad 		\mathrm{d}\omega_0(R,\cdot)\equiv0\ \ \mbox{on $TX$}\,. \]
Fix a smooth Hermitian metric $\langle\, \cdot \,|\, \cdot \,\rangle$ on $TX\otimes\mathbb{C}$ so that $T^{1,0}X$ is orthogonal to $T^{0,1}X$, $\langle\, u \,|\, v \,\rangle$ is real if $u, v$ are real tangent vectors, $\langle\,R\,|\,R\,\rangle=1$ and $R$ is orthogonal to $T^{1,0}X\oplus T^{0,1}X$. Let $(\,\cdot\,|\,\cdot\,)$ be the $L^2$ inner product on $\Omega^{0,q}(X)$ induced by $\langle\, \cdot \,|\, \cdot \,\rangle$. Let $\Box^{q}_{b}$ denote the (Gaffney extension) of the Kohn Laplacian given by \eqref{e-suIX}.
Let
\begin{equation}\label{e-suXI-Im}
S^{(q)}:L^2_{(0,q)}(X)\To{\rm Ker\,}\Box^q_b
\end{equation}
be the orthogonal projection with respect to the $L^2$ inner product $(\,\cdot\,|\,\cdot\,)$. In this work, we assume 

\begin{ass}\label{a-gue211126yyd}
The Levi form is non-degenerate of constant signature $(n_-,n_+)$, where $n_-$ denotes the 
number of negative eigenvalues of the Levi form and $n_+$ denotes the 
number of positive eigenvalues of the Levi form. We always let $q=n_-$ and suppose that 
$\Box^q_b$ has $L^2$ closed range. 
\end{ass} 

Let
$L^m_{{\rm cl\,}}(X,T^{*0,q}X\boxtimes (T^{*0,q}X)^*)$
denote the space of classical
pseudodifferential operators on $X$ of order $m$ from sections of
$T^{*0,q}X$ to sections of $T^{*0,q}X$. Let $P\in L^m_{{\rm cl\,}}(X,T^{*0,q}X\boxtimes (T^{*0,q}X)^*)$. We write $\sigma^0_P$ to denote the principal symbol of $P$. 

Let $P\in L^{0}_{{\rm cl\,}}(X,T^{*0,q}X\boxtimes(T^{*0,q}X)^*)$ be a self-adjoint classical pseudodifferential operator on $X$ with scalar principal symbol.
 The Toeplitz operator is given by 
 \begin{equation}\label{e-gue211126ycdgm}
T^{(q)}_P:=S^{(q)}\circ P\circ S^{(q)}: L^2_{(0,q)}(X)\To{\rm Ker\,}\Box^q_b
\end{equation}
and let $T^{(q)}_P(x,y)\in\mathcal{D}'(X\times X,T^{*0,q}X\boxtimes(T^{*0,q}X)^*)$ be the distribution kernel of $T^{(q)}_P$ with respect to $(\,\cdot\,|\,\cdot\,)$. 

Given a self-adjoint operator 
\[A: {\rm Dom\,}A\subset H\To H\] 
where $H$ is a Hilbert space, let $\mathrm{Spec}(A)$ denote the spectrum of $A$.  Let $\tau\in\mathcal{C}^\infty_c(\mathbb R)$ and let $\tau(A)$ denote the functional calculus of $A$. Let 
\begin{equation}\label{e-gue211201yydm}
E_{\tau}(A)={\rm Range\,}(\tau(A))\subset H.
\end{equation}
Let $I\subset\mathbb R$ be an interval. 
Consider an increasing sequence of non-negative test functions $\chi_n\,:\,\mathbb{R}\rightarrow \mathbb{R}$ such that $\mathrm{supp}(\chi_n)\Subset I$ converging pointwise to the characteristic function on $I$. By Theorem $2.5.5$ in \cite{d} the sequence of operators $\chi_n(A)$ converges strongly to a canonically determined projection $\Pi_{I}(A)$, it is called {\em spectral projection} for $I$, its range space is denoted by $E_{I}(A)$.  

For every $k\in\mathbb Z$, let $\Psi_k(X)$ denote the space of all complex Fourier integral operators of Szeg\H{o} type of order $k$ (see Definition~\ref{d-gue210624yyd}). For $A\in\Psi_k(X)$, let $\sigma^0_A$ to denote the leading term of the symbol of $A$, let $\sigma^0_{A,\mp}$ to denote the leading term of the symbol of $A_{\mp}$ (see Definition~\ref{d-gue210624yyd}).
One of the main results of this work is the following 

\begin{thm}\label{t-gue211217yyd}
Let $(X, T^{1,0}X)$ be a compact, orientable CR manifold of dimension $2n+1$, $n\geq 1$. Let $P\in L^{0}_{{\rm cl\,}}(X,T^{*0,q}X\boxtimes(T^{*0,q}X)^*)$ be a self-adjoint classical pseudodifferential operator on $X$ with scalar principal symbol and let $\tau\in\mathcal C^\infty_c(\mathbb R)$ with $0\notin{\rm supp\,}\tau$. With the same notations above, under Assumption~\ref{a-gue211126yyd}, we have $\tau(T^{(q)}_P)(x,y)\in\Psi_k(X)$, 
\begin{equation}\label{e-gue211217yyd}
\mbox{$\sigma^0_{\tau(T^{(q)}_{P}),-}(x,x)=\frac{1}{2}\pi^{-n-1}\abs{{\rm det\,}\mathcal L_x}\tau(\sigma^0_P(x,-\omega_0(x)))\pi_{x,n_-}$, for every $x\in X$},\\
\end{equation}
and if $q=n_-=n_+$, 
\begin{equation}\label{e-gue211217yydI}
\mbox{$\sigma^0_{\tau(T^{(q)}_{P}),+}(x,x)=\frac{1}{2}\pi^{-n-1}\abs{{\rm det\,}\mathcal L_x}\tau(\sigma^0_P(x,\omega_0(x)))\pi_{x,n_+}$, for every $x\in X$},
\end{equation}
where $\pi_{x,n_-}$ and $\pi_{x,n_+}$ are given by \eqref{tau140530} below and ${\rm det\,}\mathcal L_x=\lambda_1(x)\cdots\lambda_n(x)$, $\lambda_j(x)$, $j=1,\ldots,n$, are the eigenvalues of $\mathcal{L}_x$ with respect to $\langle\,\cdot\,|\,\cdot\,\rangle$. 
\end{thm}

\begin{rem}\label{r-gue211217yyd}
It should be mentioned that if $n_-=n_+$ or $\abs{n_--n_+}>1$, then $\Box^q_b$ 
has $L^2$ closed range (see~\cite{Koh86}).
\end{rem}

\subsection{Main results: Semi-classical spectral dimensions for Toeplitz operators} \label{s-gue211217yydI}

We now assume that $X$ admits a CR and transversal locally free $S^1$-action $e^{i\theta}$. Let $T$ be the vector field on $X$ given by 
\[(Tu)(x):=\frac{\pr}{\pr\theta}u(e^{i\theta}\circ x)|_{\theta=0},\ \ \forall u\in\mathcal{C}^\infty(X).\]
We take $\omega_0$ so that the associated Reeb vector field $R$ is equal to $T$. Assume that the Hermitian metric $\langle\,\cdot\,|\,\cdot\,\rangle$ on $TX \otimes\mathbb{C}$ is $S^1$-invariant. For every $m\in \mathbb{Z}$, set 
\[\Omega^{0,q}_m(X):=\set{u\in\Omega^{0,q}(X);\, (e^{i\theta})^*u=e^{im\theta}u,\ \ \forall e^{i\theta}\in S^1}.\]
 Let $L^2_{(0,q),m}(X)$ be
the completion of $\Omega^{0,q}_m(X)$ with respect to $(\,\cdot\,|\,\cdot\,)$. 
 Put 
\[({\rm Ker\,}\Box^{q}_b)_m:=({\rm Ker\,}\Box^{q}_b)\cap L^2_{(0,q),m}(X).\]
The $m$-th Szeg\H{o} projection is the orthogonal projection 
$S^{(q)}_{m}:L^2_{(0,q)}(X)\To ({\rm Ker\,}\Box^{(q)}_b)_m$
with respect to $(\,\cdot\,|\,\cdot\,)$. Let $P\in L^{0}_{{\rm cl\,}}(X,T^{*0,q}X\boxtimes(T^{*0,q}X)^*)^{S^1}$ be a self-adjoint $S^1$-invariant classical pseudodifferential operator on $X$ with scalar principal symbol (see Definition~\ref{d-gue210729yyd}). Put 
\[T^{(q)}_{P,m}:=S^{(q)}_{m}\circ P\circ S^{(q)}_m: L^2_{(0,q)}(X)\To L^2_{(0,q),m}(X).\]

We also obtain an asymptotic expansion for the functional calculus of $T^{(q)}_{P,m}$ as $m\To+\infty$ (see Theorem~\ref{t-gue211214yyds}). Moreover, we have the following semi-classical spectral dimensions for Toeplitz operators (see Section~\ref{s-gue211216yyd})

\begin{thm}\label{t-gue211217yydI}
With the notations and assumptions above, assume that $X$ is connected and $X_{{\rm reg\,}}\neq\emptyset$, where 
$X_{{\rm reg\,}}=\set{x\in X;\, e^{i\theta}x\neq x, \forall\theta\in]0,2\pi[}$. Let $P\in L^{0}_{{\rm cl\,}}(X,T^{*0,q}X\boxtimes(T^{*0,q}X)^*)^{S^1}$ be a self-adjoint $S^1$-invariant classical pseudodifferential operator on $X$ with scalar principal symbol. Let $I\subset\mathbb R$ be an open bounded interval with $0\notin\overline{I}$. We have 
\begin{equation}\label{e-gue211217yydII}
\lim_{m\To+\infty}m^{-n}{\rm dim\,}E_I(T^{(q)}_{P,m})=\frac{1}{2}\pi^{-n-1}\int_{\set{x\in X;\, \sigma^0_P(x,-\omega_0(x))\in I}}\abs{{\rm det\,}\mathcal L_x}\mathrm{dV}_X(x),
\end{equation}
where $\mathrm{dV}_X$ is the volume form on $X$ induced by $\langle\,\cdot\,|\,\cdot\,\rangle$. 
\end{thm}

\subsection{Main results: Quantization commutes with reduction for Toeplitz operators}\label{s-gue211217yydII}

In geometric quantization given a symplectic manifold $(M,\omega)$ one wants to associate a Hilbert space. The first step is called pre-quantization: if $\omega$ defines an integral cohomology class then there always exists a line bundle $L$ with Hermitian structure $h$ such that the curvature of the connection compatible with $h^L$ is $-2i\,\omega$. The second step consists in fixing a polarization, we consider complex polarizations: given an almost complex structure $J$ compatible with $\omega$, we suppose $J$ to be integrable. When $(M,\omega)$ ia a K\"ahler manifold, the \textit{Hilbert space of the quantization} is the direct sum over $k\in\mathbb{Z}$ of the spaces $H^0(M,L^k)$ of holomorphic sections of the $k$-th power of the bundle. Given a smooth function on $M$ one wants to associate an operator acting on the Hilbert space of quantization, if $(M,\omega)$ is K\"ahler, in \cite{schli} it was shown that Berezin-Toeplitz operators can be used to define a star product for the Poisson algebra of smooth functions on $M$. If $\omega$ is non-degenerate but not positive definite then the dimensions of $H^0(M,L^k)$ as $k$ goes to infinity does not grow as in the K\"ahler case, thus we consider spaces of $(0,q)$-forms. We remark that the spaces of $(0,q)$-forms we are considering here can be identified with the kernel of the Dolbeault-Dirac operator, see equation~$(2.36)$, pag.~$16$ in~\cite{Du}. In \cite{gh} we generalize the results in \cite{schli} in this setting. In this paper we are interested in studying spectrum spaces of Toeplitz operators in the presence of an action of a compact connected Lie group and we aim to generalize quantization commutes with reduction for Toeplitz operators. 

In this subsection, we do not assume that $X$ admits a transversal and CR $S^1$ action but we assume that $X$ admits a $d$-dimensional compact connected Lie group action $G$. We assume throughout that Assumption~\ref{a-gue170123I} and Assumption~\ref{a-gue170123II} hold. We refer the reader to Section~\ref{s-gue211213yyd} for some notations and terminology used here. 

Let $\mathfrak{g}$ denote the Lie algebra of $G$. For any $\xi \in \mathfrak{g}$, we write $\xi_X$ to denote the vector field on $X$ induced by $\xi$. Put $\underline{\mathfrak{g}}={\rm Span\,}(\xi_X;\, \xi\in\mathfrak{g})$,
 Let $HX=\set{{\rm Re\,}u;\, u\in T^{1,0}X}$ and let $J:HX\To HX$ be the complex structure map given by $J(u+\ol u)=iu-i\ol u$, for every $u\in T^{1,0}X$. 
Let $Y:=\mu^{-1}(0)$ and let $HY:=HX\cap TY$, where $\mu$ is the momentum map (see Definition~\ref{d-gue170124}). Fix a $G$-invariant smooth Hermitian metric $\langle\, \cdot \,|\, \cdot \,\rangle$ on $\mathbb{C}TX$ so that $T^{1,0}X$ is orthogonal to $T^{0,1}X$, $\underline{\mathfrak{g}}$ is orthogonal to $HY\cap JHY$ at every point of $Y$, $\langle \, u \,|\, v \, \rangle$ is real if $u, v$ are real tangent vectors, $\langle\,R\,|\,R\,\rangle=1$. 

Let $X_G:=\mu^{-1}(0)/G$.  In \cite{hsiaohuang}, it is proved that $X_{G}$ is a CR manifold with natural CR structure induced by $T^{1,0}X$ of dimension $2n - 2d + 1$.
Let $\mathcal{L}_{X_{G}}$ be the Levi form on $X_{G}$ induced naturally from the Levi form $\mathcal{L}$ on $X$. Since $\mathcal{L}$ is non-degenerate of constant signature, $\mathcal{L}_{X_{G}}$ is also non-degenerate of constant signature. We denote by $n_{X_G,-}$ (respectively $n_{X_G,+}$) the number of negative (respectively positive) eigenvalues of $L_{X_G}$. The Hermitian metric $\langle\,\cdot\,|\,\cdot\,\rangle$ on $TX \otimes\mathbb{C}$ induces an Hermitian metric $\langle\,\cdot\,|\,\cdot\,\rangle_{X_G}$ on $TX_G\otimes\mathbb{C}$ and let $(\,\cdot\,|\,\cdot\,)_{X_G}$ be the $L^2$ inner product on $L^2_{(0,q)}(X_G)$ induced by $\langle\,\cdot\,|\,\cdot\,\rangle_{X_G}$. We write $\mathrm{dV}_{X_G}(x)$ to denote the volume form on $X_G$ induced by $\langle\,\cdot\,|\,\cdot\,\rangle_{X_G}$.  

Let $P\in L^0_{{\rm cl\,}}(X,T^{*0,q}X\boxtimes(T^{*0,q}X)^*)^G$ be a $G$-invariant scalar pseudodifferential operator (see Definition~\ref{d-gue210729yyd}). The operator $P$ gives rise to ${P}_{X_G}$ on the CR reduction $X_G=\mu^{-1}({0})/G$. Put 
\[({\rm Ker\,}\Box^q_b)^G:={\rm Ker\,}\Box^q_b\cap L^2_{(0,q)}(X)^G,\] 
where $L^2_{(0,q)}(X)^G$ denote the space of $G$-invariant $L^2$ $(0,q)$ forms. The $G$-invariant Szeg\H{o} projection is the orthogonal projection 
\[S^{(q)}_G:L^2_{(0,q)}(X)\To ({\rm Ker\,}\Box^q_b)^G\]
with respect to $(\,\cdot\,|\,\cdot\,)$. The $G$-invariant Toeplitz operator is given by 
\[T^{(q)}_{P,G}:=S^{(q)}_G\circ P\circ S^{(q)}_G: L^2_{(0,q)}(X)\To({\rm Ker\,}\Box^q_b)^G.\]
 
 One of the main theorems of this paper is the following.

\begin{thm} \label{thm:kfourierszego}. 
	Let $X$ be a compact orientable non-degenerate CR manifold and let $G$ be a connected compact Lie group acting freely on $\mu^{-1}(0)$ such that Assumption~\ref{a-gue170123I} and Assumption~\ref{a-gue170123II} hold. Let $P\in L^0_{{\rm cl\,}}(X,T^{*0,q}X\boxtimes(T^{*0,q}X)^*)^G$ be a $G$-invariant scalar pseudodifferential operator and assume that $\sigma^0_P(x,\mp\omega_0(x))>0$ for every $x\in X$.  Thus, $\mathrm{Spec}(T^{(n_{\mp})}_{P,G})\subset I\cup\set{0}$  if $\Box^{n_{\mp}}_b$ has closed range, $\mathrm{Spec}(T^{(n_{X_G,\mp})}_{P_{X_G}})\subset I\cup\set{0}$, if $\Box^{n_{X_G,\mp}}_{b,X_G}$ has closed range, for some open bounded interval $I$ with $0\notin\overline{I}$ (see Lemma~\ref{l-gue211126yydq}). Recall that we work with Assumption~\ref{a-gue211126yyd}. 
	
If $n_-=n_+$, $n_{X_G,-}=n_{X_G,+}$ or $n_-\neq n_+$, $n_{X_G,-}\neq n_{X_G,+}$. Suppose that $\Box^{n_{X_G,-}}_b$ has closed range. 
The map 
\[\sigma: E_{I}(T^{(n_-)}_{P,G})\To E_{I}(T^{(n_{X_G,-})}_{P_{X_G}})\]
given by \eqref{e-gue180308II} and \eqref{e-gue211213yydr} below is Fredholm. 

If $n_-=n_+$ and $n_{X_G,-}\neq n_{X_G,+}$. Suppose that $\Box^{n_{X_G,-}}_b$ and $\Box^{n_{X_G,+}}_b$ have closed range. 
The map 
\[\sigma: E_{I}(T^{(n_-)}_{P,G})\To E_{I}(T^{(n_{X_G,-})}_{P_{X_G}})\oplus E_{I}(T^{(n_{X_G,+})}_{P_{X_G}})\]
given by \eqref{e-gue180308IIa} below is Fredholm. 

If $n_-\neq n_+$ and $n_{X_G,-}=n_{X_G,+}$. Suppose that $\Box^{n_{+}}_b$ and $\Box^{n_{X_G,-}}_b$ have closed range. (Recall that we always assume that $\Box^{n_{-}}_b$ has closed range.)
The map 
\[\sigma: E_{I}(T^{(n_-)}_{P,G})\oplus E_{I}(T^{(n_+)}_{P,G})\To E_{I}(T^{(n_{X_G,-})}_{P_{X_G}})\]
given by \eqref{e-gue211214yyd} below is Fredholm.
\end{thm} 

We refer the reader to Theorem~\ref{t-gue211213yydI} and Theorem~\ref{t-gue211214yyd} for more general setting of Theorem~\ref{thm:kfourierszego}. 

Let $H^{q}_b(X)^G={\rm Ker\,}\Box^q_b\cap L^2_{(0,q)}(X)^G$, $H^{q}_b(X_G):={\rm Ker\,}\Box^q_{b,X_G}$.  We have ${\rm Spec\,}(S^{(q)})\subset\set{0,1}$. Take $\tau(x)=1$ near $x=1$. 
As a corollary of Theorem~\ref{thm:kfourierszego}, we deduce 

\begin{cor}\label{c-gue211213yydam}
With the same notations and assumptions above, if $n_-=n_+$, $n_{X_G,-}=n_{X_G,+}$ or $n_-\neq n_+$, $n_{X_G,-}\neq n_{X_G,+}$, the map
\[\sigma: H^{n_-}_b(X)^G\To H^{n_{X_G,-}}_b(X_G)\]
given by \eqref{e-gue180308II} and \eqref{e-gue211213yydr} below is Fredholm. 

If $n_-=n_+$ and $n_{X_G,-}\neq n_{X_G,+}$. Suppose that $\Box^{n_{X_G,-}}_b$ and $\Box^{n_{X_G,+}}_b$ have closed range. 
The map 
\[\sigma: H^{n_-}_b(X)^G\To H^{n_{X_G,-}}_b(X_G)\oplus H^{n_{X_G,+}}_b(X_G)\]
given by \eqref{e-gue180308IIa} below is Fredholm. 

If $n_-\neq n_+$ and $n_{X_G,-}=n_{X_G,+}$. Suppose that $\Box^{n_{+}}_b$ and $\Box^{n_{X_G,-}}_b$ have closed range. 
The map 
\[\sigma: H^{n_-}_b(X)^G\oplus H^{n_+}_b(X)^G\To H^{n_{X_G,-}}_b(X_G)\]
given by \eqref{e-gue180308IIa} below is Fredholm. 
\end{cor}

Theorem~\ref{thm:kfourierszego} has a natural application in geometric quantization. Given a symplectic $(M,\omega)$ with $[\omega] \in H^2(M,\mathbb{Z})$ one can always consider an almost complex structure $J$ compatible with $\omega$. Suppose that $J$ is integrable then $(M,\omega,J)$ is a complex manifold. Let $(L,h^L)$ be the quantizing holomorphic line bundle over $M$ and let $(L^k,h^{L^k})$ be the $k$-th power of $(L,h^L)$, where $h^L$ denotes the Hermitian metric of $L$. Let $R^L=-2i\,\omega$ be the curvature of $L$ induced by $h^L$. Fix a Hermitian metric $\langle\,\cdot\,|\,\cdot\,\rangle$ on the holomorphic tangent bundle $T^{1,0}M$ of $M$ and let $(\,\cdot\,|\,\cdot)_k$ be the $L^2$ inner product of $\Omega^{0,q}(M,L^k)$ induced by $\langle\,\cdot\,|\,\cdot\,\rangle$  and $h^{L^k}$, where $\Omega^{0,q}(M,L^k)$ denotes the space of smooth $(0,q)$ forms of $M$ with values in $L^k$. Let 
\[\Box^q_k:=\ddbar^*\,\ddbar+\ddbar\,\ddbar^*: \Omega^{0,q}(M,L^k)\To \Omega^{0,q}(M,L^k)\]
be the Kodaira Laplacian, where $\ddbar^*$ is the adjoint of $\ddbar$ with respect to $(\,\cdot\,|\,\cdot)_k$.
Let 
\[B^{(q)}_k: L^2_{(0,q)}(M,L^k)\To{\rm Ker\,}\Box^q_k\]
be the orthogonal projection with respect to $(\,\cdot\,|\,\cdot)_k$ (Bergman projection). Let $f\in\mathcal{C}^\infty(M)$, the Toeplitz operator is given by 
\[T_{f,k}^{(q)}:=B^{(q)}_k\circ M_f\circ B^{(q)}_k: L^2_{(0,q)}(M,L^k)\To{\rm Ker\,}\Box^q_k,\]
where $M_f$ denote the operator given by the multiplication $f$. 
In the presence of a holomorphic and Hamiltonian action on $(M,\omega)$ of a compact connected Lie group $G$, suppose that $0\in \mathfrak{g}^*$ is a regular value for the symplectic momentum map and the action of $G$ is free near $\mu^{-1}(0)$. We can consider the Marsden–Weinstein reduction $(M_{G},\omega_G)$ which is a quantizable symplectic manifold. Assume that  the  Hermitian metric $\langle\,\cdot\,|\,\cdot\,\rangle$ is $G$-invariant as in CR case and the Hermitian metric $\langle\,\cdot\,|\,\cdot\,\rangle$ on $T^{1,0}M$ induces a Hermitian metric $\langle\,\cdot\,|\,\cdot\,\rangle_{M_G}$ on $T^{1,0}M_G$. Assume that the action $G$ can be lifted to $L$ and $h^L$ is $G$-invariant. Then, $L_G:=L/G$ is a holomorphic line bundle over $M_G$. 
Fix a $G$-invariant smooth function $f$ on $M$, let
\[B^{(q)}_{G,k}: L^2_{(0,q)}(M,L^k)\To({\rm Ker\,}\Box^q_k)^G\]
be the orthogonal projection with respect to $(\,\cdot\,|\,\cdot)_k$ ($G$-invariant Bergman projection), where $({\rm Ker\,}\Box^q_k)^G$ is the space of $G$-invariant holomorphic sections. We can define two Toeplitz operators. The first is the $G$-invariant Toeplitz operator
\[T_{f,G,k}^{(q)}:=B^{(q)}_{G,k}\circ M_f\circ B^{(q)}_{G,k}: L^2_{(0,q)}(M,L^k)\To ({\rm Ker\,}\Box^q_k)^G.\]
The second is the standard Toeplitz operator $T_{f_{M_G},k}^{(q)}$ on the quotient $M_G$. 
Applying Theorem~\ref{thm:kfourierszego} to the circle bundle of $(L^*, h^{L^*})$, we get 

\begin{thm} \label{thm:circleaction}
	With the same notations and assumptions above, suppose that the curvature $R^L$ is non-degenerate of constant signature $(n_-,n_+)$  and that the curvature $R^{L_G}$ is non-degenerate of constant signature $(n_{M_G,-},n_{M_G,+})$, where $R^{L_G}$ is the curvature of $L_G$ induced by the $G$-invariant Hermitian metric $h^L$ on $L$.
	Let $f\in\mathcal{C}^\infty(M)$ be a $G$-invariant positive function. Then, $\mathrm{Spec}(T^{(n_{\mp})}_{f,G,k})\subset I\cup\set{0}$, $\mathrm{Spec}(T^{(n_{M_G,\mp})}_{f_{M_G}})\subset I\cup\set{0}$, for every $k\in\mathbb Z$, for some open bounded interval $I$ with $0\notin\overline{I}$.
	
	If $n_-=n_+$ and $n_{M_G,-}=n_{M_G,+}$ or $n_-\neq n_+$ and $n_{X_G,-}\neq n_{X_G,+}$, then
\[E_I(T^{n_-}_{f,G,k})\cong E_I(T^{n_{M_G,-}}_{f_{M_G},k})\]
for $\abs{k}\gg1$. 

If $n_-=n_+$ and $n_{M_G,-}\neq n_{M_G,+}$, then 
\[E_I(T^{n_-}_{f,G,k})\cong E_I(T^{n_{M_G,-}}_{f_{M_G},k})\oplus E_I(T^{n_{M_G,+}}_{f_{M_G},k})\]
for $\abs{k}\gg1$. 

If $n_-\neq n_+$ and $n_{X_G,-}=n_{X_G,+}$, then 
\[E_I(T^{n_-}_{f,G,k})\oplus E_I(T^{n_+}_{f,G,k})\cong E_I(T^{n_{M_G,-}}_{f_{M_G},k})\]
	for $\abs{k}\gg1$. 
\end{thm} 

From Theorem~\ref{t-gue211217yydI} and Theorem~\ref{t-gue211214yyd} (general version of Theorem~\ref{thm:kfourierszego}), we obtain the semi-classical spectral dimensions for $G$-invariant Toeplitz operators on complex manifolds (see Theorem~\ref{t-gue211215yydm})

\begin{thm}\label{t-gue211217ycd}
With the same notations and assumptions above, suppose that the curvature $R^L$ is non-degenerate of constant signature $(n_-,n_+)$  and that the curvature $R^{L_G}$ is non-degenerate of constant signature $(n_{M_G,-},n_{M_G,+})$, where $R^{L_G}$ is the curvature of $L_G$ induced by the $G$-invariant Hermitian metric $h^L$ on $L$.
Let $I\subset\mathbb R$ be an open bounded interval with $0\notin\overline{I}$.
We have 
\begin{equation}\label{e-gue211217ycdq}
\lim_{k\To+\infty}k^{-n+d}{\rm dim\,}E_I(T^{(n_-)}_{f,G,k})=(2\pi)^{-n+d}\int_{\set{x\in M_G;\, f(x)\in I}}\abs{{\rm det\,}R^{L_G}}\mathrm{dV}_{M_G}(x),
\end{equation}
where $d={\rm dim\,}G$, ${\rm det\,}R^{L_G}(x)=\lambda_1(x)\cdots\lambda_d(x)$, $\lambda_j(x)$, $j=1,\ldots,d$, are the eigenvalues of 
$R^{L_G}(x)$ with respect to $\langle\,\cdot\,|\,\cdot\,\rangle_{M_G}$  and $\mathrm{dV}_{M_G}(x)$ is the volume form on $M_G$ induced by $\langle\,\cdot\,|\,\cdot\,\rangle_{M_G}$. 
\end{thm}

\begin{rem}\label{r-gue211217ycda}
(i) In Theorem~\ref{t-gue211217yydI} and Theorem~\ref{t-gue211217ycd}, $I$ is any open interval with $0\notin\overline{I}$. We do not need the assumptions about $I$ as in Theorem~\ref{thm:kfourierszego} and Theorem~\ref{thm:circleaction}. Moreover, in Theorem~\ref{t-gue211217yydI} and Theorem~\ref{t-gue211217ycd}, we do not need the assumptions that $\sigma^0_P(x,\mp\omega_0(x))>0$ for every $x\in X$ and $f$ is positive. 

(ii) In Theorem~\ref{t-gue211215yydm}, we establish the semi-classical spectral dimensions for $G$-invariant Toeplitz operators on CR manifolds with locally free $S^1$ action. Moreover, in Theorem~\ref{t-gue211215yydm}, we can replace the $G$-invariant smooth function $f\in\mathcal{C}^\infty(M)$ by any $G\times S^1$-invariant scalar pseudodifferential operator $P$. 
\end{rem}

\section{Preliminaries}\label{s:prelim}

\subsection{Standard notations} \label{s-ssna}
We use the following notations through this article: $\mathbb N=\{1,2,\ldots\}$ is the set of natural numbers, $\mathbb N_0=\mathbb N\cup\{0\}$, $\mathbb R$ is the set of 
real numbers, $\mathbb R_+=\{x\in\mathbb R;\, x>0\}$, $\overline{\mathbb R}_+=\{x\in\mathbb R;\, x\geq0\}$. 
We write $\alpha=(\alpha_1,\ldots,\alpha_n)\in\mathbb N^n_0$ 
if $\alpha_j\in\mathbb N_0$, 
$j=1,\ldots,n$. 

Let $M$ be a $\mathcal{C}^\infty$ paracompact manifold.
We let $TM$ and $T^*M$ denote the tangent bundle of $M$
and the cotangent bundle of $M$, respectively.
The complexified tangent bundle of $M$ and the complexified cotangent bundle of $M$ will be denoted by $TM\otimes \Complex$
and $T^*M\otimes \Complex$, respectively. Write $\langle\,\cdot\,,\cdot\,\rangle$ to denote the point-wise
duality between $TM$ and $T^*M$.
We extend $\langle\,\cdot\,,\cdot\,\rangle$ bi-linearly to $TM\otimes \Complex\times T^*M\otimes \Complex$.
Let $B$ be a $\mathcal C^\infty$ vector bundle over $M$. The fibre of $B$ at $x\in M$ will be denoted by $B_x$.
Let $E$ be a vector bundle over a $\mathcal{C}^\infty$ paracompact manifold $N$. We write
$B\boxtimes E^*$ to denote the vector bundle over $M\times N$ with fibre over $(x, y)\in M\times N$
consisting of the linear maps from $E_y$ to $B_x$.  Let $Y\subset M$ be an open set. 
From now on, the spaces of distribution sections of $B$ over $Y$ and
smooth sections of $B$ over $Y$ will be denoted by $\mathcal D'(Y, B)$ and $\mathcal{C}^\infty(Y, B)$, respectively.
Let $\mathcal E'(Y, B)$ be the subspace of $\mathcal D'(Y, B)$ whose elements have compact support in $Y$.
Let $\mathcal{C}^\infty_c(Y,B):=\mathcal{C}^\infty(Y,B)\cap\mathcal E'(Y, B)$. 
For $m\in\Real$, let $H^m(Y, B)$ denote the Sobolev space of order $m$ of sections of $B$ over $Y$. Put
\begin{gather*}
H^m_{\rm loc\,}(Y, B)=\left\{u\in\mathcal{D}'(Y, B);\, \varphi u\in H^m(Y, B),
   \, \forall\varphi\in\mathcal{C}^\infty_c(Y)\right\}\,,\\
     H^m_{\rm comp\,}(Y, B)=H^m_{\rm loc}(Y, B)\cap\mathcal E'(Y, B)\,.
\end{gather*}

We recall the Schwartz kernel theorem.
Let $B$ and $E$ be $\mathcal{C}^\infty$ vector
bundles over paracompact orientable $\mathcal{C}^\infty$ manifolds $M$ and $N$, respectively, equipped with smooth densities of integration. If
$A: \mathcal{C}^\infty_c(N,E)\To\mathcal D'(M,B)$
is continuous, we write $A(x, y)$ to denote the distribution kernel of $A$.
The following two statements are equivalent
\begin{enumerate}
	\item $A$ is continuous: $\mathcal E'(N,E)\To\mathcal{C}^\infty(M,B)$,
	\item $A(x,y)\in\mathcal{C}^\infty(M\times N,B\boxtimes E^*)$.
\end{enumerate}
If $A$ satisfies (1) or (2), we say that $A$ is smoothing on $M \times N$. Let
$A,\hat A:\mathcal{C}^\infty_c(N,E)\To\mathcal{D}'(M,B)$ be continuous operators.
We write 
\begin{equation} \label{e-gue160507f}
	\mbox{$A\equiv \hat A \quad$ on $M\times N$} 
\end{equation}
if $A-\hat A$ is a smoothing operator. If $M=N$, we simply write ``on $M$". 
We say that $A$ is properly supported if the restrictions of the two projections 
$(x,y)\mapsto x$, $(x,y)\mapsto y$ to ${\rm supp\,}(A(x,y))$
are proper.

Let $H(x,y)\in\mathcal D'(M\times N,B\boxtimes E^*)$. We write $H$ to denote the unique 
continuous operator $\mathcal C^\infty_c(N,E)\To\mathcal D'(M,B)$ with distribution kernel $H(x,y)$. 
In this work, we identify $H$ with $H(x,y)$.

\subsection{CR manifolds}

We recall some notations concerning CR and contact geometry. Let $(X, T^{1,0}X)$ be a compact, orientable CR manifold of dimension $2n+1$, $n\geq 1$, where $T^{1,0}X$ is a CR structure of $X$. There is a unique sub-bundle $HX$ of $TX$ such that $HX\otimes \mathbb{C}=T^{1,0}X \oplus T^{0,1}X$, which is called horizontal tangent bundle. Let $J:HX\To HX$ be the complex structure map given by $J(u+\ol u)=\imath u-\imath\ol u$, for every $u\in T^{1,0}X$; thus by complex linear extension of $J$ to $TX\otimes\mathbb{C}$, the $\imath$-eigenspace of $J$ is $T^{1,0}X$. 

Since $X$ is orientable, there always exists a real non-vanishing $1$-form $\omega_0\in\mathcal{C}^{\infty}(X,T^*X)$ so that $\langle\,\omega_0(x)\,,\,u\,\rangle=0$, for every $u\in H_xX$, for every $x\in X$. If $d\omega_0|_{HX}$ is non-degenerate, $\omega_0$ is a contact form. For each $x \in X$, we define a quadratic form on $HX$ by
\[\mathcal{L}_x(U,V) =\frac{1}{2}\mathrm{d}\omega_0(JU, V),\qquad\forall \ U, V \in H_xX.\]
Then, we extend $\mathcal{L}$ to $HX\otimes\mathbb{C}$ by complex linear extension; for $U, V \in T^{1,0}_xX$,
\begin{equation}\label{e-gue211216ycdp}
\mathcal{L}_x(U,\overline{V}) = \frac{1}{2}\,\mathrm{d}\omega_0(JU, \overline{V}) = -\frac{1}{2i}\,\mathrm{d}\omega_0(U,\overline{V}).
\end{equation}
The Hermitian quadratic form $\mathcal{L}_x$ on $T^{1,0}_xX$ is called Levi form at $x$. In this work, we assume that the Levi form is non-degenerate of constant signature 
$(n_-,n_+)$. The Reeb vector field $R\in\mathcal{C}^\infty(X,TX)$ is defined to be the non-vanishing vector field determined by 
\[	\omega_0(R)=-1,\quad 		\mathrm{d}\omega_0(R,\cdot)\equiv0\ \ \mbox{on $TX$}. \]

Fix a smooth Hermitian metric $\langle\, \cdot \,|\, \cdot \,\rangle$ on $TX\otimes \mathbb{C}$ so that $T^{1,0}X$ is orthogonal to $T^{0,1}X$, $\langle\, u \,|\, v \,\rangle$ is real if $u, v$ are real tangent vectors, $\langle\,R\,|\,R\,\rangle=1$ and $R$ is orthogonal to $T^{1,0}X\oplus T^{0,1}X$. For $u \in TX\otimes \mathbb{C}$, we write $|u|^2 := \langle\, u\, |\, u\, \rangle$. Denote by $T^{*1,0}X$ and $T^{*0,1}X$ the dual bundles of $T^{1,0}X$ and $T^{0,1}X$, respectively. They can be identified with sub-bundles of the complexified cotangent bundle $T^*X\otimes\mathbb{C}$. For $q\in\mathbb N$, the bundle of $(0,q)$ forms of $X$ is given by $T^{*0,q}X:=\Lambda^q(T^{*0,1}X)$. The Hermitian metric 
$\langle\, \cdot \,|\, \cdot \,\rangle$ on $TX\otimes\mathbb{C}$ induces a Hermitian metric $\langle\, \cdot \,|\, \cdot \,\rangle$ on $T^{*0,q}X$. 

Let $D$ be an open set of $X$. Let $\Omega^{0,q}(D)$ denote the space of smooth sections of $T^{*0,q}X$ over $D$ and let $\Omega^{0,q}_c(D)$ be the subspace of $\Omega^{0,q}(D)$ whose elements have compact support in $D$. Let 
\[
\ddbar_b:\Omega^{0,q}(X)\To\Omega^{0,q+1}(X)
\]
be the tangential Cauchy-Riemann operator. Let $\mathrm{d}v(x)$ be the volume form on $X$ induced by the Hermitian metric $\langle\,\cdot\,|\,\cdot\,\rangle$.
The $L^2$-inner product $(\,\cdot\,|\,\cdot\,)$ on $\Omega^{0,q}(X)$ 
induced by $\mathrm{d}v(x)$ and $\langle\,\cdot\,|\,\cdot\,\rangle$ is given by
\[
(\,u\,|\,v\,):=\int_X\langle\,u(x)\,|\,v(x)\,\rangle\, \mathrm{d}v(x)\,,\quad u,v\in\Omega^{0,q}(X)\,.
\]
We denote by $L^2_{(0,q)}(X)$ 
the completion of $\Omega^{0,q}(X)$ with respect to $(\,\cdot\,|\,\cdot\,)$. 
We extend $(\,\cdot\,|\,\cdot\,)$ to $L^2_{(0,q)}(X)$ 
in the standard way. For every $f\in L^2_{(0,q)}(X)$, we denote by $\norm{f}^2:=(\,f\,|\,f\,)$.
We extend
$\ddbar_{b}$ to $L^2_{(0,r)}(X)$, $r=0,1,\ldots,n$, by
\[
\ddbar_{b}:{\rm Dom\,}\ddbar_{b}\subset L^2_{(0,r)}(X)\To L^2_{(0,r+1)}(X)\,,
\]
where ${\rm Dom\,}\ddbar_{b}:=\{u\in L^2_{(0,r)}(X);\, \ddbar_{b}u\in L^2_{(0,r+1)}(X)\}$ and, 
for any $u\in L^2_{(0,r)}(X)$, $\ddbar_{b} u$ is defined in the sense of distributions. 
We also write
\[
\ol{\pr}^{*}_{b}:{\rm Dom\,}\ol{\pr}^{*}_{b}\subset L^2_{(0,r+1)}(X)\To L^2_{(0,r)}(X)
\]
to denote the Hilbert space adjoint of $\ddbar_{b}$ in the $L^2$ space with respect to the inner product $(\,\cdot\,|\,\cdot\, )$.
Let $\Box^{q}_{b}$ denote the so called Gaffney extension of the Kohn Laplacian which is given by
\begin{equation}\label{e-suIX}
\renewcommand{\arraystretch}{1.2}
\begin{array}{c}
{\rm Dom\,}\Box^q_{b}=\left\{s\in L^2_{(0,q)}(X);\, 
s\in{\rm Dom\,}\ddbar_{b}\cap{\rm Dom\,}\ol{\pr}^{*}_{b},\,
\ddbar_{b}s\in{\rm Dom\,}\ol{\pr}^{*}_{b}, \, \ol{\pr}^{*}_{b}s\in{\rm Dom\,}\ddbar_{b}\right\}\,,\\
\Box^{q}_{b}s=\ddbar_{b}\ol{\pr}^{*}_{b}s+\ol{\pr}^{*}_{b}\ddbar_{b}s
\:\:\text{for $s\in {\rm Dom\,}\Box^q_{b}$}\,.
 \end{array}
\end{equation}
Eventually, let
\begin{equation}\label{e-suXI-I}
S^{(q)}:L^2_{(0,q)}(X)\To{\rm Ker\,}\Box^q_b
\end{equation}
be the orthogonal projection with respect to the $L^2$-inner product $(\,\cdot\,|\,\cdot\,)$ and let
\[
S^{(q)}(x,y)\in\mathcal{D}'(X\times X,T^{*0,q}X\boxtimes(T^{*0,q}X)^*)
\]
denote the distribution kernel of $S^{(q)}$, which is called Szeg\H{o} kernel for $(0,q)$-forms. 

\section{Functional calculus for Toeplitz operators}\label{s-gue211214yydI}

In this section, we will show that the functional calculus for Toeplitz operators are complex Fourier integral operators. In the following section we introduce some symbol spaces. 

\subsection{Some symbol spaces}\label{s-gue210914yyd}

First, we recall H\"ormander symbol spaces. Let $D\subset X$ be a local coordinate patch with local coordinates $x=(x_1,\ldots,x_{2n+1})$. 

\begin{defn}\label{d-gue211126yyd}
	For $m\in\Real$, $S^m_{1,0}(D\times D\times\mathbb{R}_+,T^{*0,q}X\boxtimes(T^{*0,q}X)^*)$ 
	is the space of all $$a(x,y,t)\in\mathcal{C}^\infty(D\times D\times\mathbb{R}_+,T^{*0,q}X\boxtimes(T^{*0,q}X)^*)$$ 
	such that, for all compact $K\Subset D\times D$ and all $\alpha, \beta\in\mathbb N^{2n+1}_0$, $\gamma\in\mathbb N_0$, 
	there is a constant $C_{\alpha,\beta,\gamma}>0$ such that 
	\[\abs{\pr^\alpha_x\pr^\beta_y\pr^\gamma_t a(x,y,t)}\leq C_{\alpha,\beta,\gamma}(1+\abs{t})^{m-\gamma},\ \ 
	\mbox{for every $(x,y,t)\in K\times\Real_+, t\geq1$}.\]
	Furthermore, put 
	\[
	S^{-\infty}(D\times D\times\mathbb{R}_+,T^{*0,q}X\boxtimes(T^{*0,q}X)^*) :=\bigcap_{m\in\Real}S^m_{1,0}(D\times D\times\mathbb{R}_+,T^{*0,q}X\boxtimes(T^{*0,q}X)^*).
	\]
	We sometimes simply write $S^{m}_{1,0}$ to denote $S^m_{1,0}(D\times D\times\mathbb{R}_+,T^{*0,q}X\boxtimes(T^{*0,q}X)^*)$, $m\in\mathbb R\cup\set{-\infty}$. 
	Let $a_j\in S^{m_j}_{1,0}$, $j=0,1,2,\ldots$ with $m_j\To-\infty$, as $j\To\infty$. Then there exists $a\in S^{m_0}_{1,0}$ 	unique modulo $S^{-\infty}$, such that 	$a-\sum^{k-1}_{j=0}a_j\in S^{m_k}_{1,0}$ for $k=0,1,2,\ldots$. 
	
	If $a$ and $a_j$ have the properties above, we write $$\mbox{$a\sim\sum^{\infty}_{j=0}a_j$ in $S^{m_0}_{1,0}\left(D\times D\times\mathbb{R}_+,T^{*0,q}X\boxtimes(T^{*0,q}X)^*\right)\,.$}$$ 
	We write
	\[
	s(x, y, t)\in S^{m}_{{\rm cl\,}}(D\times D\times\mathbb{R}_+,T^{*0,q}X\boxtimes(T^{*0,q}X)^*)
	\]
	if $s(x, y, t)\in S^{m}_{1,0}(D\times D\times\mathbb{R}_+,T^{*0,q}X\boxtimes(T^{*0,q}X)^*)$ and 
	\begin{align*}
		s(x, y, t)\sim\sum^\infty_{j=0}s^j(x, y)t^{m-j}\text{ in }S^{m}_{1, 0}
		(D\times D\times\mathbb{R}_+\,,T^{*0,q}X\boxtimes(T^{*0,q}X)^*)\,,
	\end{align*}
	where $s^j(x, y)\in\mathcal{C}^\infty(D\times D,T^{*0,q}X\boxtimes(T^{*0,q}X)^*),\ j\in\N_0$.
\end{defn} 


Let $D\subset X$ be an open set. Let
$L^m_{{\rm cl\,}}(D,T^{*0,q}X\boxtimes (T^{*0,q}X)^*)$
denote the space of classical
pseudodifferential operators on $D$ of order $m$ from sections of
$T^{*0,q}X$ to sections of $T^{*0,q}X$. Let $P\in L^m_{{\rm cl\,}}(D,T^{*0,q}X\boxtimes (T^{*0,q}X)^*)$. We write $\sigma^0_P$ to denote the principal symbol of $P$.

Let $P\in L^0_{{\rm cl\,}}(D,T^{*0,q}X\boxtimes (T^{*0,q}X)^*)$ be a self-adjoint operator with positive scalar principal symbol. 
Let $\tau\in\mathcal{C}^\infty_c(I)$, $I\subset\mathbb R$ is an open interval. In the next sections we are going to study $\tau(T^{(q)}_{P})$, where $\tau(T^{(q)}_{P})$ denotes the functional calculus of $T^{(q)}_{P}$. Here $T^{(q)}_P$ denotes the Toeplitz operator with symbol $P$ (see \eqref{e-gue211126ycdg}). Thus, we need to introduce another symbol spaces. Let $U \Subset \mathbb{C}$ be an open set such with $U\cap \mathbb{R}\neq \emptyset$. Let $\hat U:=\{z\in U;\, {\rm Im\,}z\neq0\}$. 

\begin{defn}\label{d-gue140221a}
	For $m\in\Real$, the space $S^m_{z}(D\times D\times\mathbb{R}_+,T^{*0,q}X\boxtimes(T^{*0,q}X)^*)$ 
	is the set of all $a(x,y,t,z)\in\mathcal{C}^\infty(D\times D\times\mathbb{R}_+\times \hat{U},T^{*0,q}X\boxtimes(T^{*0,q}X)^*)$ such that for every $\alpha,\,\beta \in \mathbb{N}^{2n+1}_0$, $\gamma\in\mathbb N_0$ there exists $N(\alpha,\,\beta,\,\gamma)\in\mathbb{N}$ such that for every compact set $K\Subset D\times D$, we have
	\[\abs{\pr^\alpha_x\pr^\beta_y\pr^\gamma_t a(x,y,t,z)}\leq C_{K}\abs{\mathrm{Im}z}^{-N(\alpha,\beta,r)}(1+\abs{t})^{m-\gamma},\ \ 
	\mbox{for every $(x,y,t,z)\in K\times\Real_+\times\hat U, t\geq1$},\]
	where $C_{K}>0$ is a constant independent of $z$.  
\end{defn} 

\begin{defn}\label{d-gue140221b}
	With the notations used above, the space $\hat{\mathcal{C}}^\infty(D\times D\times\hat{U},T^{*0,q}X\boxtimes(T^{*0,q}X)^*)$ is  the set of all $f\in\mathcal{C}^\infty(D\times D\times\hat{U},T^{*0,q}X\boxtimes(T^{*0,q}X)^*)$ such that for every $\alpha,\,\beta \in \mathbb{N}^{2n+1}_0$ there exists $N(\alpha,\,\beta)\in\mathbb{N}$ such that for every compact set $K\Subset D\times D$, we have
	\[\abs{\pr^\alpha_x\pr^\beta_y f(x,y,z)}\leq C_{K}\abs{\mathrm{Im}z}^{-N(\alpha,\beta)},\ \ 
	\mbox{for every $(x,y,z)\in K\times\hat U$},\]
	where $C_{K}>0$ is a constant independent of $z$.  
	
	We define the space $\hat{\mathcal{C}}^\infty(X\times X\times\hat{U},T^{*0,q}X\boxtimes(T^{*0,q}X)^*)$ in the similar way. 
\end{defn} 

\begin{rem}\label{r-gue211201yyd}
We identify $D$ with an open set of $\mathbb R^{2n+1}$. Let $D^{\mathbb C}$ be an open set of $\mathbb C^{2n+1}$ with $D^{\mathbb C}\cap\mathbb R^{2n+1}=D$. 
We can also define the space $\hat{\mathcal{C}}^\infty(D^{\mathbb C}\times D^{\mathbb C}\times\hat{U},T^{*0,q}X\boxtimes(T^{*0,q}X)^*)$ in the same way. From now on, for any $f\in\mathcal{C}^\infty(D\times D\times\hat{U},T^{*0,q}X\boxtimes(T^{*0,q}X)^*)$, we always take an almost analytic extension $\Td f$ of $f$ so that 
$\Td f\in\hat{\mathcal{C}}^\infty(D^{\mathbb C}\times D^{\mathbb C}\times\hat{U},T^{*0,q}X\boxtimes(T^{*0,q}X)^*)$ (it is straightforward to see that this is always possible). 
\end{rem}

\begin{defn} \label{d-gue211126yyda}
With the notations used above, fix $m\in\mathbb{R}$. For every $$a\in S^m_{z}(D\times D\times\mathbb{R}_+,T^{*0,q}X\boxtimes(T^{*0,q}X)^*)$$ we say that 
	$$a\in S^m_{z,\mathrm{cl}}(D\times D\times\mathbb{R}_+,T^{*0,q}X\boxtimes(T^{*0,q}X)^*)$$
	if we can find $a_j\in \hat{\mathcal{C}}^\infty(D\times D\times \hat{U},T^{*0,q}X\boxtimes(T^{*0,q}X)^*)$, $j=0,\,1,\,\dots$, such that
	\begin{equation}\label{e-gue211126yyds}
	a(x,\,y,\,t,\,z)- \sum^N_{j=0}a_j(x,\,y,\,z)\,t^{m-j}\in S^{m-N-1}_{z}(D\times D\times\mathbb{R}_+,T^{*0,q}X\boxtimes(T^{*0,q}X)^*)
	\end{equation}
	for all $N\in \mathbb{N}$.
\end{defn} 

In the same way as for H\"ormander symbol spaces, put 
\[
S^{-\infty}_z(D\times D\times\mathbb{R}_+,T^{*0,q}X\boxtimes(T^{*0,q}X)^*) :=\bigcap_{m\in\Real}S^m_{z}(D\times D\times\mathbb{R}_+,T^{*0,q}X\boxtimes(T^{*0,q}X)^*).
\]
\begin{lem}
With the notations used above, let $a_j\in S^{m_j}_{z,\,\mathrm{cl}}(D\times D\times\mathbb{R}_+,T^{*0,q}X\boxtimes(T^{*0,q}X)^*)$, 
	$j=0,1,2,\ldots$ with $m_j\To-\infty$, as $j\To\infty$. 
	Then there exists 
	$$a\in S^{m_0}_{z,\,\mathrm{cl}}(D\times D\times\mathbb{R}_+,T^{*0,q}X\boxtimes(T^{*0,q}X)^*)$$
	unique modulo $S^{-\infty}_z$, such that 
	\begin{equation} \label{eq: aaj}
		a-\sum^{k-1}_{j=0}a_j\in S^{m_k}_{z,\,\mathrm{cl}}\left(D\times D\times\mathbb{R}_+,T^{*0,q}X\boxtimes(T^{*0,q}X)^*\right)
	\end{equation}
	for $k=1,2,\ldots$. If $a$ and $a_j$ have the properties above, we write $$\mbox{$a\sim\sum^{\infty}_{j=0}a_j$ in $S^{m_0}_{z,\,\mathrm{cl}}\left(D\times D\times\mathbb{R}_+,T^{*0,q}X\boxtimes(T^{*0,q}X)^*\right)$ \,.}$$ 
\end{lem}

\begin{proof}
	Let $K_0 \Subset K_1  \Subset\cdots  \Subset D$, $\cup_{j=1}^{\infty}K_j=D$. For every $j=0,1,2,\dots$, take $0<\epsilon_j< 1$ small enough and $N_j\in \mathbb{N}$ so that
	\begin{equation}\label{eq:aj} \lVert a_j(x,y,t,z)(1-\chi(t\lvert \mathrm{Im}z \rvert^{N_j}\epsilon^{N_j}_j)) \rVert_{\mathcal C^j(K_j\times K_j)} \leq 2^{-j}\,t^{m_j+1-j}\,,
	\end{equation}
	for every $z\in \hat{U}$, $t\geq 1$, where $\chi \in \mathcal{C}^{\infty}_c(\mathbb{R})$, $\chi \equiv 1$ on $[-1,\,1]$. Let
	\[a(x,y,t,z)= \sum_{j=0}^{+\infty} a_j(x,y,t,z)(1-\chi(t\lvert \mathrm{Im}z \rvert^{N_j}\epsilon^{N_j}_j))\,. \]
	From \eqref{eq:aj} we can check that $a$ is well defined as a smooth function on $D\times D\times \mathbb{R}_+\times \hat{U}$, $a\in \hat{S}^{m_0}_{z,{\rm cl\,}}$ and satisfies \eqref{eq: aaj}.
\end{proof}

\subsubsection{Fourier integral operators of Szeg\H{o} type}

We are now going to recall results about Toeplitz operators relying on \cite{gh}. In this work, we work with Assumption~\ref{a-gue211126yyd}.


 Let us adopt the following notation
\begin{equation*}
	\Sigma=\Sigma^-\cup\Sigma^+,\,\Sigma^-=\set{(x,\lambda\omega_0(x))\in T^*X;\,\lambda<0},\,\Sigma^+=\set{(x,\lambda\omega_0(x))\in T^*X;\,\lambda>0}.
\end{equation*}
It is known that (see~\cite[Theorem 1.2]{hsiao},~\cite[Theorem 4.7]{HM14}) there exist continuous operators
$S_-, S_+: L^2_{(0,q)}(X)\To{\rm Ker\,}\Box^q_b$ such that 
\[S^{(q)}=S_-+S_+,\ \ S^*_-=S^*_-,\ \ S^*_+=S^*_+,\,S_-=S^{(q)}S_-,\ \ S_+=S^{(q)}S_+,\]
\[	{\rm WF'\,}(S_-)={\rm diag\,}(\Sigma^-\times\Sigma^-),S_+\equiv0\ \ \mbox{if $q\neq n_+$},\,
	{\rm WF'\,}(S_+)={\rm diag\,}(\Sigma^+\times\Sigma^+)\ \ \mbox{if $q=n_-=n_+$},
\]
where $S^*_{\mp}$ is the adjoint of $S_{\mp}$, ${\rm WF\,}(S_{\mp})$ is the wave front set of $S_{\mp}$ in the sense of H\"ormander. We have the following (see~\cite[Theorem 1.2]{hsiao},~\cite[Theorem 4.7]{HM14}) about micro-local properties of the Szeg\H{o} kernel.

\begin{thm}\label{t-gue161109I}
	Let $D\subset X$ be any local coordinate patch with coordinates $x=(x_1,\ldots,x_{2n+1})$, then 
	$S_-(x,y)$, $S_+(x,y)$ satisfy
	\[S_{\mp}(x, y)\equiv\int^{\infty}_{0}e^{i\varphi_{\mp}(x, y)t}s_{\mp}(x, y, t)\mathrm{d}t\ \ \mbox{on $D$},\]
	with 
	\begin{equation}  \label{e-gue161110r}
		s_{\mp}(x,y,t)\sim\sum^\infty_{j=0}s^j_{\mp}(x, y)t^{n-j}\in S^{n}_{1, 0},\,\text{ where }
		s^j_{\mp}(x, y)\in\mathcal{C}^\infty(D\times D,T^{*0,q}X\boxtimes(T^{*0,q}X)^*),
	\end{equation}
	$s_+(x,y,t)=0$ if $q\neq n_+$ and $s^0_-(x,x)\neq0$, for all $x\in D$. The phase functions $\varphi_-$, $\varphi_+$  satisfy
	\[
	\renewcommand{\arraystretch}{1.2}
	\begin{array}{ll}
		&\varphi_+, \varphi_-\in\mathcal{C}^\infty(D\times D),\ \ {\rm Im\,}\varphi_{\mp}(x, y)\geq0,
		\varphi_-(x, x)=0,\ \ \varphi_-(x, y)\neq0\ \ \mbox{if}\ \ x\neq y,\\
		&\mathrm{d}_x\varphi_-(x, y)\big|_{x=y}=-\omega_0(x), \ \ \mathrm{d}_y\varphi_-(x, y)\big|_{x=y}=\omega_0(x), 
		-\ol\varphi_+(x, y)=\varphi_-(x,y).
	\end{array}
	\]
\end{thm}

The following result describes the phase function in local coordinates (see chapter 8 of part I in \cite{hsiao}).

\begin{thm} \label{t-gue161110g}
	For a given point $p\in X$, let $\{W_j\}_{j=1}^{n}$
	be an orthonormal frame of $T^{1, 0}X$ in a neighborhood of $p$
	such that
	the Levi form is diagonal at $p$, i.e.\ $\mathcal{L}_{p}(W_{j},\overline{W}_{s})=\delta_{j,s}\mu_{j}$, $j,s=1,\ldots,n$.
	We take local coordinates $x=(x_1,\ldots,x_{2n+1})$, $z_j=x_{2j-1}+ix_{2j}$, $j=1,2,\ldots,n$, defined on some neighborhood of $p$ such that $\omega_0(p)=\mathrm{d}x_{2n+1}$, $x(p)=0$, and, for some $c_j\in\Complex$, $j=1,\ldots,n$\,,
	\begin{equation}\label{e-gue161219a}
		R=-\frac{\pr}{\pr x_{2n+1}},\quad
		W_j=\frac{\pr}{\pr z_j}-i\mu_j\ol z_j\frac{\pr}{\pr x_{2n+1}}-
		c_jx_{2n+1}\frac{\pr}{\pr x_{2n+1}}+\sum^{2n}_{k=1}a_{j,k}(x)\frac{\pr}{\pr x_k}+O(\abs{x}^2)\,,
	\end{equation}
	where $j=1,\ldots,n$, $a_{j,k}\in\mathcal{C}^\infty(X)$, $a_{j,k}(x)=O(\abs{x})$, for every $j=1,\ldots,n$, $k=1,\ldots,2n$. 
	Set
	$y=(y_1,\ldots,y_{2n+1})$, $w_j=y_{2j-1}+iy_{2j}$, $j=1,2,\ldots,n$.
	Then, for $\varphi_-$ in Theorem~\ref{t-gue161109I}, we have
	\begin{equation} \label{e-gue140205VI}
		{\rm Im\,}\varphi_-(x,y)\geq c\sum^{2n}_{j=1}\abs{x_j-y_j}^2,\ \ c>0,
	\end{equation}
	in some neighborhood of $(0,0)$ and
	\begin{equation} \label{e-gue140205VII}
		\renewcommand{\arraystretch}{1.2}
		\begin{array}{l}
			\varphi_-(x, y)=-x_{2n+1}+y_{2n+1}+i\sum^{n}_{j=1}\abs{\mu_j}\abs{z_j-w_j}^2 \\
			\ +\sum^{n}_{j=1}\left(i\mu_j(\ol z_jw_j-z_j\ol w_j)+c_j(-z_jx_{2n+1}+w_jy_{2n+1})+\ol c_j(-\ol z_jx_{2n+1}+\ol w_jy_{2n+1})\right)\\
			\ +(x_{2n+1}-y_{2n+1})f(x, y) +O(\abs{(x, y)}^3),
		\end{array}
	\end{equation}
	where $f$ is smooth and satisfies $f(0,0)=0$.
\end{thm} 

We pause and introduce some notations in order to given an explicit description for the leading term of the symbol of Szeg\H{o} kernel. For a given point $p\in X$, let $\{W_j\}_{j=1}^{n}$ be an
orthonormal frame of $(T^{1,0}X,\langle\,\cdot\,|\,\cdot\,\rangle)$ near $p$, for which the Levi form
is diagonal at $p$. We will denote by
\begin{equation}\label{det140530}
	\mathcal{L}_{p}(W_j,\ol W_\ell)=\mu_j(p)\delta_{j,\ell}\,,\;\; j,\ell=1,\ldots,n\,\quad\text{ and }\quad \det\mathcal{L}_{p}=\prod_{j=1}^{n}\mu_j(p)\,.
\end{equation}
Let $\{e_j\}_{j=1}^{n}$ denote the basis of $T^{*0,1}X$, dual to $\{\ol W_j\}^{n}_{j=1}$. We assume that, if $1\leq j\leq n_-$, $\mu_j(p)<0$ and if $n_-+1\leq j\leq n$ we have $\mu_j(p)>0$. Put
\[\mathcal{N}(p,n_-):=\set{ce_1(p)\wedge\ldots\wedge e_{n_-}(p);\, c\in\Complex},\,
\mathcal{N}(p,n_+):=\set{ce_{n_-+1}(p)\wedge\ldots\wedge e_{n}(p);\, c\in\Complex}
\]
and let
\begin{equation}\label{tau140530}
	\pi_{p,n_-}:T^{*0,q}_{p}X\To\mathcal{N}(p,n_-)\,,\quad
	\pi_{p,n_+}:T^{*0,q}_{p}X\To\mathcal{N}(p,n_+)\,,
\end{equation}
be the orthogonal projections onto $\mathcal{N}(p,n_-)$ and $\mathcal{N}(p,n_+)$
with respect to $\langle\,\cdot\,|\,\cdot\,\rangle$, respectively. For $J=(j_1,\ldots,j_q)$, $1\leq j_1<\cdots<j_q\leq n$, let 
$e_J:=e_{j_1}\wedge\cdots\wedge e_{j_q}$. For $\abs{I}=\abs{J}=q$, $I$, $J$ are strictly increasing, let $e_I\otimes(e_J)^*$ be the linear transformation from 
$T^{*0,q}X$ to $T^{*0,q}X$ given by 
\[(e_I\otimes(e_J)^*)(e_K)=\delta_{J,K}e_I,\]
for every $\abs{K}=q$, $K$ is strictly increasing, where $\delta_{J,K}=1$ if $J=K$, $\delta_{J,K}=0$ if $J\neq K$. For any $f\in T^{*0,q}X\boxtimes(T^{*0,q}X)^*$, we have \[f=\sideset{}{'}\sum_{\abs{I}=\abs{J}=q}c_{I,J}e_I\otimes(e_J)^*,\]
$c_{I,J}\in\mathbb C$, for all $\abs{I}=\abs{J}=q$, $I$, $J$ are strictly increasing, where $\sum'$ means that the summation is performed only over strictly increasing multi-indices.  We call $c_{I,J}e_I\otimes(e_J)^*$ the component of $f$ in the direction $e_I\otimes(e_J)^*$. Let $I_0=(1,2,\ldots,q)$. We can check that 
\[\pi_{p,n_-}=e_{I_0}(p)\otimes(e_{I_0}(p))^*.\]

The following formula for the leading term $s^0_-$ on the diagonal follows from \cite[\S 9]{hsiao}. The formula for the leading term $s^0_+$ on the diagonal follows similarly.

\begin{thm} \label{t-gue140205III}
	For the leading term $s^0_{-}(x,y)$ of the expansion \eqref{e-gue161110r} of $s_{-}(x,y,t)$, we have
	\[
	s^0_{-}(x_0, x_0)=\frac{1}{2}\pi^{-n-1}\abs{\det\mathcal{L}_{x_0}}\pi_{x_0,n_{-}}\,,\:\:x_0\in D.
	\]
\end{thm} 

We need to recall the following definition from \cite{gh}.

\begin{defn}\label{d-gue210624yyd}
	Let $H: \Omega^{0,q}(X)\To\Omega^{0,q}(X)$ be a continuous operator with distribution kernel $H(x,y)\in\mathcal{D}'(X\times X,T^{*0,q}X\boxtimes(T^{*0,q}X)^*)$. 
	We say that $H$ is a complex Fourier integral operator of Szeg\H{o} type of order $k\in\mathbb Z$ if $H$ is smoothing away the diagonal on $X$ and 
	for every local coordinate patch $D\subset X$ with local coordinates $x=(x_1,\ldots,x_{2n+1})$, we have on $D$
	\[H(x,y)\equiv H_-(x,y)+H_+(x,y)\,,\quad H_{\mp}(x,y)\equiv\int^\infty_0e^{i\varphi_{\mp}(x, y)t}a_{\mp}(x, y, t)\mathrm{d}t \,,\]
	where $a_-, a_+\in S^{k}_{{\rm cl\,}}(D\times D\times\mathbb{R}_+,T^{*0,q}X\boxtimes(T^{*0,q}X)^*)$, $a_+=0$ if $q\neq n_+$, $\varphi_-$, $\varphi_+$ are as in 
	Theorem~\ref{t-gue161109I} and Theorem~\ref{t-gue161110g}. We write $\sigma^0_{H,-}(x,y)$ to denote the leading term of the expansion \eqref{e-gue161110r} of $a_{-}(x,y,t)$. If $q=n_+$, we write $\sigma^0_{H,+}(x,y)$ to denote the leading term of the expansion \eqref{e-gue161110r} of $a_{+}(x,y,t)$. If $n_-\neq n_+$, we sometimes write $\sigma^0_H(x,y):=\sigma^0_{H,-}(x,y)$. Note that $\sigma^0_{H,-}(x,y)$ and $\sigma^0_{H,+}(x,y)$ depend on the choices of the phases $\varphi_-$ and $\varphi_+$ but $\sigma^0_{H,-}(x,x)$ and $\sigma^0_{H,+}(x,x)$ are independent of the choices of the phases $\varphi_-$ and $\varphi_+$.
	
	Let us denote by $\Psi_k(X)$ the space of all complex Fourier integral operators of Szeg\H{o} type of order $k$. 
\end{defn}

Let $U\Subset\mathbb C$ be an open set with $U\cap\mathbb R\neq\emptyset$. Let $\hat U:=\{z\in U;\, {\rm Im\,}z\neq0\}$. We need 

\begin{defn}\label{d-gue211126ycd}
With the notations used above, let $H_z: \Omega^{0,q}(X)\To\Omega^{0,q}(X)$ be a $z$-dependent continuous operator with distribution kernel $H_z(x,y)\in\mathcal{D}'(X\times X,T^{*0,q}X\boxtimes(T^{*0,q}X)^*)$, $z\in\hat U$. We write
\[H_z\equiv0\mod\hat{\mathcal{C}}^\infty\]
or \[H_z(x,y)\equiv0\mod\hat{\mathcal{C}}^\infty\]
if 
\[H_z(x,y)\in\hat{\mathcal{C}}^\infty(X\times X\times\hat U,T^{*0,q}X\boxtimes(T^{*0,q}X)^*).\]

If $H_z\equiv0\mod\hat{\mathcal{C}}^\infty$, we say that $H_z$ is $z$-smoothing.

We say that $H_z$ is $z$-smoothing away the diagonal if for every $\chi_1, \chi_2\in\mathcal{C}^\infty_c(X)$ with 
\[{\rm supp\,}\chi_1\cap{\rm supp\,}\chi_2=\emptyset,\] we have $\chi_1H_z\chi_2\equiv0\mod\hat{\mathcal{C}}^\infty$. 

Let $A_z: \Omega^{0,q}(X)\To\Omega^{0,q}(X)$ be another $z$-dependent continuous operator with distribution kernel $A_z(x,y)\in\mathcal{D}'(X\times X,T^{*0,q}X\boxtimes(T^{*0,q}X)^*)$, $z\in\hat U$. 
We write
\[H_z\equiv A_z\mod\hat{\mathcal{C}}^\infty\]
or \[H_z(x,y)\equiv A_z(x,y)\mod\hat{\mathcal{C}}^\infty\]
if $H_z-A_z\equiv0\mod\hat{\mathcal{C}}^\infty$. In this definition, we can replace $X$ by any open set $D$ of $X$. 
\end{defn}

Similarly we define

\begin{defn}\label{typez}
	With the notations used above, let $H_z: \Omega^{0,q}(X)\To\Omega^{0,q}(X)$ be a $z$-dependent continuous operator with distribution kernel $H_z(x,y)\in\mathcal{D}'(X\times X,T^{*0,q}X\boxtimes(T^{*0,q}X)^*)$, $z\in\hat U$. 
	We say that $H_z$ is a complex Fourier integral operator of $z$-Szeg\H{o} type of order $k\in\mathbb Z$ if $H_z$ is $z$-smoothing away the diagonal on $X$ and 
	for every local coordinate patch $D\subset X$ with local coordinates $x=(x_1,\ldots,x_{2n+1})$, we have on $D$
	\[H_z(x,y)\equiv H_{z,-}(x,y)+H_{z,+}(x,y)\mod\hat{\mathcal{C}}^\infty,\quad H_{z,\mp}(x,y)\equiv\int^\infty_0e^{i\varphi_{\mp}(x, y)t}a_{\mp}(x, y, t,z)\mathrm{d}t \,,\]
	where $a_-, a_+\in S^{k}_{z,{\rm cl\,}}(D\times D\times\mathbb{R}_+,T^{*0,q}X\boxtimes(T^{*0,q}X)^*)$, $a_+=0$ if $q\neq n_+$, $\varphi_-$, $\varphi_+$ are as in 
	Theorem~\ref{t-gue161109I} and Theorem~\ref{t-gue161110g}. We write $\sigma^0_{H_z,-}(x,y)$ to denote the leading term of the expansion \eqref{e-gue211126yyds} of $a_{-}(x,y,t,z)$. If $q=n_+$, we write $\sigma^0_{H_z,+}(x,y)$ to denote the leading term of the expansion \eqref{e-gue211126yyds} of $a_{+}(x,y,t,z)$. If $n_-\neq n_+$, we sometimes write $\sigma^0_{H_z}(x,y):=\sigma^0_{H_z,-}(x,y)$.
	
	Let us denote by $\Psi_{k,z}(X)$ the space of all complex Fourier integral operators of $z$-Szeg\H{o} type of order $k$. 
\end{defn}

The following is straightforward. We omit the details. 

\begin{lem}\label{l-gue211126yyd}
With the notations used above, let $D$ be a local coordinate patch of $X$ with local coordinates $x=(x_1,\ldots,x_{2n+1})$. Let 
\[A_z(x,y)=\int^\infty_0e^{i\varphi_{-}(x, y)t}t^kb(x, y, z)\mathrm{d}t,\]
where $k\in\mathbb Z$, $b(x,y,z)\in\hat{\mathcal{C}}^\infty(D\times D\times\hat U,T^{*0,q}X\boxtimes(T^{*0,q}X)^*)$. If $b$ vanishes to infinite order at $x=y$, then 
\[A_z\equiv0\mod\hat{\mathcal{C}}^\infty.\]
\end{lem}



From Lemma~\ref{l-gue211126yyd}, we can repeat  the proof of  Lemma $3.1$ in \cite{gh} and get 

\begin{lem}\label{l-gue211126ycdc}
Let $B_z\in\Psi_{k,z}(X)$ with $S^{(q)}B_z\equiv B_z\mod\hat{\mathcal C}^\infty$,  $B_zS^{(q)}\equiv B_z\mod\hat{\mathcal C}^\infty$. If	 
	\[\pi_{x,n_-}\sigma^0_{B_z,-}(x,x)\pi_{x,n_-}=0
	\text{ and } 
	\pi_{x,n_+}\sigma^0_{B_z,+}(x,x)\pi_{x,n_+}=0\text{ if }q=n_-=n_+\,,\]
	for every $x\in X$, then $B_z\in\Psi_{k-1,z}(X)$. 
\end{lem} 

We need 

\begin{lem}\label{l-gue211205yyd}
Let $A_z\in\Psi_{n-1,z}(X)$.There exist $N\in\mathbb N$ and $c>0$ such that
\[\norm{A_zu}\leq\frac{c}{\abs{{\rm Im\,}z}^N}\norm{u},\]
for every $z\in\hat U$ and every $u\in\Omega^{0,q}(X)$. 
\end{lem}

\begin{proof}
Let $A^*_z: \Omega^{0,q}(X)\To\Omega^{0,q}(X)$ be the adjoint of $A_z$ with respect to $(\,\cdot\,|\,\cdot\,)$. For every $u\in\Omega^{0,q}(X)$, we have 
\begin{equation}\label{e-gue211205yyd}
\begin{split}
&\norm{A_zu}^2=(\,A^*_zA_zu\,|\,u\,)\leq\norm{A^*_zA_zu}\norm{u}\\
&\leq\norm{(A^*_zA_z)^2u}^{\frac{1}{2}}\norm{u}^{\frac{3}{2}}\leq\cdots\leq
\norm{(A^*_zA_z)^{2^k}u}^{\frac{1}{2^k}}\norm{u}^{2-\frac{1}{2^k}},
\end{split}
\end{equation}
for every $k\in\mathbb N$ and every $z\in\hat U$. From complex stationary phase formula, we see that 
\[(A^*_zA_z)^{2^k}\in\Psi_{n-2^{k+1},z}(X).\]
We take $k\gg1$ so that 
\begin{equation}\label{e-gue211205yydI}
\norm{(A^*_zA_z)^{2^k}v}\leq\frac{c}{\abs{{\rm Im\,}z}^{N_1}}\norm{v},
\end{equation}
for every $z\in\hat U$ and every $v\in\Omega^{0,q}(X)$, where $c>0$ and $N_1\in\mathbb N$ are constants. From \eqref{e-gue211205yyd} 
and \eqref{e-gue211205yydI}, the lemma follows. 
\end{proof}

\begin{lem}\label{l-gue211205yydI}
Let $A_z\in\Psi_{n,z}(X)$. For every $s\in\mathbb Z$, there exist $N_s\in\mathbb N$ and $c>0$ such that
\begin{equation}\label{e-gue211205yydII}
\norm{A_zu}_{s-1}\leq\frac{c}{\abs{{\rm Im\,}z}^{N_s}}\norm{u}_s,
\end{equation}
for every $z\in\hat U$ and every $u\in\Omega^{0,q}(X)$, where $\norm{\cdot}_s$ denotes the Sobolev norm of order $s$ on $X$.  

Thus, $A_z$ can be extended to a continuous operator: 
\[A_z: H^s(X,T^{*0,q}X)\To H^{s-1}(X,T^{*0,q}X)\]
and we have the estimates \eqref{e-gue211205yydII}, for every $s\in\mathbb Z$, every $z\in\hat U$  and every $u\in H^s(X)$.
\end{lem}

\begin{proof}
For every $s\in\mathbb Z$, let $\Lambda_s\in L^s_{{\rm cl\,}}(X,T^{*0,q}X\boxtimes (T^{*0,q}X)^*)$
be a classical elliptic
pseudodifferential operator on $X$ of order $s$ from sections of
$T^{*0,q}X$ to sections of $T^{*0,q}X$. From complex stationary phase formula, we see that 
\[\Lambda_{s-1}A_z\Lambda_{-s}\in\Psi_{n-1,z}(X).\]
From this observation and Lemma~\ref{l-gue211205yyd}, the lemma follows. 
\end{proof}

Let $P\in L^{\ell}_{{\rm cl\,}}(X,T^{*0,q}X\boxtimes(T^{*0,q}X)^*)$ with scalar principal symbol, $\ell\leq0$, $\ell\in\mathbb Z$. We recall that a Toeplitz operator is given by \begin{equation}\label{e-gue211126ycdg}
T^{(q)}_P:=S^{(q)}\circ P\circ S^{(q)}: L^2(X)\To{\rm Ker\,}\Box^q_b. 
\end{equation}
Let $f\in\mathcal{C}^\infty(X)$ and let $M_f$ denote the operator given by the multiplication $f$. When $P=M_f$, we write $T^{(q)}_f:=T^{(q)}_P$.

The following follows from the standard calculus of Fourier integral operator of complex type (see the calculation after Theorem $3.4$ in \cite{gh}). 

\begin{thm}\label{t-gue210701yyd}
	Let $P\in L^{\ell}_{{\rm cl\,}}(X,T^{*0,q}X\boxtimes(T^{*0,q}X)^*)$ with scalar principal symbol, $\ell\leq0$, $\ell\in\mathbb Z$. We have $T^{(q)}_P\in\Psi_{n+\ell}(X)$ and 
	\[\mbox{$\sigma^0_{T^{(q)}_{P,-}}(x,x)=\frac{1}{2}\pi^{-n-1}\abs{{\rm det\,}\mathcal L_x}\sigma^0_P(x,-\omega_0(x))\pi_{x,n_-}$, for every $x\in X$}.\]
	If $q=n_-=n_+$, then
	\[\mbox{$\sigma^0_{T^{(q)}_{P,+}}(x,x)=\frac{1}{2}\pi^{-n-1}\abs{{\rm det\,}\mathcal L_x}\sigma^0_P(x,\omega_0(x))\pi_{x,n_+}$, for every $x\in X$}.\]
\end{thm}

\subsubsection{The spectral projections and functional calculus for Toeplitz operators}
\label{sec:specproj}

Given a self-adjoint operator 
\[A: {\rm Dom\,} A\subset H\To H\] 
where $H$ is a Hilbert space, let $\mathrm{Spec}(A)$ denote the spectrum of $A$.  Let $\tau\in\mathcal{C}^\infty_c(\mathbb R)$ and let $\tau(A)$ denote the functional calculus of $A$. Let 
\begin{equation}\label{e-gue211201yyd}
E_{\tau}(A)={\rm Range\,}(\tau(A))\subset H.
\end{equation}
Let $I\subset\mathbb R$ be an interval. 
Consider an increasing sequence of non-negative test functions $\chi_n\,:\,\mathbb{R}\rightarrow \mathbb{R}$ such that $\mathrm{supp}(\chi_n)\Subset I$ converging point-wise to the characteristic function on $I$. By Theorem $2.5.5$ in \cite{d} the sequence of operators $\chi_n(A)$ converges strongly to a canonically determined projection $\Pi_{I}(A)$, it is called {\em spectral projection} for $I$, its range space is denoted by $E_{I}(A)$.

Let $P\in L^{0}_{{\rm cl\,}}(X,T^{*0,q}X\boxtimes(T^{*0,q}X)^*)$ be a self-adjoint classical pseudodifferential operator on $X$ with scalar principal symbol. Then, it is not difficult to see that the Toeplitz operator $T^{(q)}_P$ is self-adjoint. In this section, we will show that the functional calculus for Toeplitz operators is a complex Fourier integral operators of Szeg\H{o} type of order $0$. 

\begin{lem}\label{l-gue211126yydq} 
Suppose that $\sigma^0_P(x,\mp\omega_0(x))\neq0$ for every $x\in X$.
We have $\mathrm{Spec}(T^{(q)}_P)\subset I\cup\set{0}$, for some open interval $I$ with $0\notin\overline I$.
\end{lem} 

\begin{proof}
	The statement is equivalent to show that $a<\lVert T^{(q)}_Pu\rVert<b$, for every $u\in L^2_{(0,q)}(X)$, $u\perp{\rm Ker\,}T^{(q)}_P$, $\lVert u\rVert =1$. Since $P$ is of order zero, $T^{(q)}_P$ is $L^2$ bounded, there is a $b>0$ such that $
	\Vert T^{(q)}_Pu \rVert <b$, for all $u\in L^2_{(0,q)}(X)$, $\lVert u\rVert =1$. We only need to show that there is a $c>0$ such that $\Vert T^{(q)}_Pu \rVert \geq c$ for all $u\in L^2_{(0,q)}(X)$, $u\perp{\rm Ker\,}T^{(q)}_P$, $\lVert u\rVert =1$. If this is not true, we can find  $u_j\in L^2_{(0,q)}(X)$, $u_j\perp{\rm Ker\,}T_P$, $\lVert u_j\rVert =1$, $j=1,2,\dots$, such that $\Vert T^{(q)}_Pu_j \rVert< 1/j$, $j=1,2,\dots$. Let $v_j:=T^{(q)}_P u_j$, $j=1,2,\ldots$, then $\lim_j \lVert v_j\rVert=0$. 
	Since $\sigma^0_P(x,\mp\omega_0(x))\neq0$ for every $x\in X$, by using the same computation as in \cite{gh} we can find a pseudodifferential operator $Q$ of order $0$ such that 
	\[T^{(q)}_Q T^{(q)}_P= S^{(q)}+F,\]
	where $F$ is smoothing. Thus, $T^{(q)}_Q T^{(q)}_P u_j=T^{(q)}_Q v_j = S^{(q)}u_j+Fu_j$, $j=1,2,\dots$. Since $u_j\perp{\rm Ker\,}T^{(q)}_P$, $u_j = S^{q}u_j$, hence $T^{(q)}_Q v_j = u_j +Fu_j$, $j=1,2,\dots$. Since $T^{(q)}_Q$ is $L^2$ bounded, $\lim_j \lVert T^{(q)}_Q v_j \rVert=0$. Fix $\epsilon >0$. By Rellich's Lemma, there is a subsequence $1\leq j_1<j_2<\dots$ such that $u_{j_s}\rightarrow u$ in $H^{-\epsilon}(X)$ as $s\rightarrow +\infty$. Since $F$ is smoothing, $F:H^{-\epsilon}(X)\rightarrow L^2(X)$ is continuous, $\lim_s F u_{j_s}=u$ in $L^2_{(0,q)}(X)$. Thus, $\lim_su_{j_s}=u$ in $L^2_{(0,q)}(X)$. Since $u_j\perp{\rm Ker\,}T^{(q)}_P$ for each $j$, then $u\perp {\rm Ker\,}T^{(q)}_P$; but $T^{(q)}_P u=\lim_s T^{(q)}_P u_{j_s}=0$ and thus we get a contradiction. 
\end{proof}

Now, take $\tau\in\mathcal C^\infty_c(I)$, where $I\subset\mathbb R_+$ is an open interval with $\ol I$ is a compact subset of $\mathbb R_+$. 
Let $\Td\tau\in\mathcal C^\infty_c(I^{\mathbb C})$ be an almost analytic extension of $\tau$, where $I^{\mathbb C}$ is a bounded interval of $\mathbb C$ with $I^{\mathbb C}\cap\mathbb R=I$. The explicit expression for the operator $\tau(T^{(q)}_P)$ due to Heffler and Sj\"ostrand is the following: 
\begin{equation}\label{eq:chinh} 
\tau(T^{(q)}_P)=-\frac{1}{\pi}\int_{\mathbb{C}}\frac{\partial \tilde{\tau}(z)}{\partial \bar{z}} (z-T^{(q)}_P)^{-1} \mathrm{d}x\,\mathrm{d}y,
\end{equation}
where $z=x+iy$. From \eqref{eq:chinh}, we shall study the distributional kernel of $(z-T_P)^{-1}$. 

For simplicity, from now on, until further notice, we always assume that $n_-\neq n_+$. We will always assume that $z\in I^{\mathbb{C}}$ with $\mathrm{Im}z\neq0$. 

\begin{thm} \label{lem:Tgmf} 
We can find $B_z\in \Psi_{n,z}(X)$ so that 
\[(z-T^{(q)}_P)B_z\equiv S^{(q)}\mod\hat{\mathcal{C}}^\infty,\] 
$B_z\equiv S^{(q)}B_z\equiv B_zS^{(q)}$ and 
	\begin{align}\label{e-gue21212yydc} \pi_{x,n_-}\sigma^0_{B_z}(x,x,z)\pi_{x,n_-}=\frac{1}{z-\sigma^0_P(x,-\omega_0(x))}\cdot s_-^0(x,x)\,, 
	\end{align}
	for every $x\in X$, where $s^0_-(x,x)$ is as in Theorem~\ref{t-gue140205III}. 
\end{thm} 

\begin{proof}
 Let $x=(x_1,\ldots,x_{2n+1})$ be local coordinates of $X$ defined on an open set $D\subset X$. We will first work on $D$. 
Let $\hat B_{0,z}(x,y):=\int^{+\infty}_0e^{it\varphi_-(x,y)}\hat b_0(x,y,z,t)\mathrm{d}t$, where 
\[\hat b_0(x,y,z,t):=\frac{s^0_-(x,x)}{z-\sigma^0_P(x,-\omega_0(x))}t^n\in S^n_{z,{\rm cl\,}}(D\times D\times\mathbb R_+,T^{*0,q}X\boxtimes(T^{*0,q}X)^*).\]
Put $B_{0,z}:=S^{(q)}\circ\hat B_{0,z}\circ S^{(q)}$. From complex stationary phase formula of Melin-Sj\"ostrand, we have 
$B_{0,z}\in\Psi_{n,z}(X)$ and 
\[\begin{split}
&B_{0,z}(x,y)\equiv\int^{+\infty}_0e^{it\varphi_-(x,y)}b_0(x,y,z,t)\mathrm{d}t\mod\hat{\mathcal{C}}^\infty,\\
&\mbox{$b_0(x,y,z,t)\sim\sum^{+\infty}_{j=0}b_{0,j}(x,y,z)t^{n-j}$ in $S^n_{z,{\rm cl\,}}(D\times D\times\mathbb R_+,T^{*0,q}X\boxtimes(T^{*0,q}X)^*)$}
\end{split}\]
and 
\begin{equation}\label{e-gue211203ycd}
b_{0,0}(x,x,z)=\frac{s^0_-(x,x)}{z-\sigma^0_P(x,-\omega_0(x))}t^n.
\end{equation}
From complex stationary phase formula of Melin-Sj\"ostrand and \eqref{e-gue211203ycd}, it is straightforward to check that 
\begin{equation}\label{e-gue211203ycdI}
\begin{split}
&(z-T^{(q)}_P)B_{0,z}-S^{(q)}\equiv\int^{+\infty}_0e^{it\varphi_-(x,y)}\hat r_0(x,y,z,t)\mathrm{d}t\mod\hat{\mathcal{C}}^\infty,\\
&\mbox{$\hat r_0(x,y,z,t)\sim\sum^{+\infty}_{j=0}\hat r_{0,j}(x,y,z)t^{n-j}$ in $S^n_{z,{\rm cl\,}}(D\times D\times\mathbb R_+,T^{*0,q}X\boxtimes(T^{*0,q}X)^*)$},\\
&\mbox{$\hat r_{0,0}(x,x,z)=0$, for every $x\in D$}.
\end{split}\end{equation}
From Lemma~\ref{l-gue211126ycdc} and \eqref{e-gue211203ycdI}, we see that $(z-T^{(q)}_P)B_{0,z}-S^{(q)}\in\Psi_{n-1,z}(D)$. Thus, 
\begin{equation}\label{e-gue211203ycdII}\begin{split}
&(z-T^{(q)}_P)B_{0,z}-S^{(q)}\equiv\int^{+\infty}_0e^{it\varphi_-(x,y)}r_0(x,y,z,t)\mathrm{d}t\mod\hat{\mathcal{C}}^\infty,\\
&\mbox{$r_0(x,y,z,t)\sim\sum^{+\infty}_{j=0}r_{0,j}(x,y,z)t^{n-1-j}$ in $S^{n-1}_{z,{\rm cl\,}}(D\times D\times\mathbb R_+,T^{*0,q}X\boxtimes(T^{*0,q}X)^*)$},\\
&\mbox{$r_{0,0}(x,x,z)=0$, for every $x\in D$}.
\end{split}\end{equation}
Let 
\[\begin{split}
&\hat G_z(x,y):=\int^{+\infty}_0e^{it\varphi_-(x,y)}\hat b_1(x,y,z,t)\mathrm{d}t\in\Psi_{n-1,z}(D),\\
&\hat b_1(x,y,z,t):=-\pi_{x,n_-}r_{0,0}(x,x,z)\pi_{x,n_-}\Bigr(\frac{1}{2}\abs{{\rm det\,}\mathcal{L}_x}\pi^{-n-1}\Bigr)^{-1}\frac{t^{n-1}}{z-\sigma^0_P(x,-\omega_0(x))}.
\end{split}\]
Let $G_z:=S^{(q)}\circ\hat G_z\circ S^{(q)}$. From complex stationary phase formula, we have 
\begin{equation}\label{e-gue211203ycda}
\begin{split}
&G_z(x,y)\equiv\int^{+\infty}_0e^{it\varphi_-(x,y)}b_1(x,y,z,t)\mathrm{d}t\mod\hat{\mathcal{C}}^\infty,\\
&b_1(x,y,z,t)=-\pi_{x,n_-}r_{0,0}(x,x,z)\pi_{x,n_-}\frac{t^{n-1}}{z-\sigma^0_P(x,-\omega_0(x))}.
\end{split}\end{equation}
Let $B_{1,z}:=B_{0,z}+G_z$. From \eqref{e-gue211203ycdII} and \eqref{e-gue211203ycda}, we have 
\begin{equation}\label{e-gue211203ycdIIa}\begin{split}
&(z-T^{(q)}_P)B_{1,z}-S^{(q)}\equiv\int^{+\infty}_0e^{it\varphi_-(x,y)}r_1(x,y,z,t)\mathrm{d}t\mod\hat{\mathcal{C}}^\infty,\\
&\mbox{$r_1(x,y,z,t)\sim\sum^{+\infty}_{j=0}r_{1,j}(x,y,z)t^{n-1-j}$ in $S^{n-1}_{z,{\rm cl\,}}(D\times D\times\mathbb R_+,T^{*0,q}X\boxtimes(T^{*0,q}X)^*)$},\\
&\mbox{$\pi_{x,n_-}r_{1,0}(x,x,z)\pi_{x,n_-}=0$, for every $x\in D$}.
\end{split}\end{equation}
From Lemma~\ref{l-gue211126ycdc} and \eqref{e-gue211203ycdIIa}, we see that $(z-T^{(q)}_P)B_{1,z}-S^{(q)}\in\Psi_{n-2,z}(D)$. 
Continuing in this way, we get $b_j(x,y,z,t)\in S^{n-j}_{z,{\rm cl\,}}(D\times D\times\mathbb R_+,T^{*0,q}X\boxtimes(T^{*0,q}X)^*)$, $j=0,1,\ldots$, such that for every $N\in\mathbb N$, we have 
\[(z-T^{(q)}_P)B^N_{1,z}-S^{(q)}\in\Psi_{n-N-1,z}(D),\]
where $B^N_z:=\int^{+\infty}_0e^{it\varphi_-(x,y)}\sum^N_{j=0}b_j(x,y,z,t)\mathrm{d}t$. Let 
\[\mbox{$b(x,y,z,t)\sim\sum^{+\infty}_{j=0}b_j(x,y,z,t)$ in $S^{n}_{z,{\rm cl\,}}(D\times D\times\mathbb R_+,T^{*0,q}X\boxtimes(T^{*0,q}X)^*)$}\]
and put 
\[\Td B_{z,D}:=\int^{+\infty}_0e^{it\varphi_-(x,y)}b(x,y,z,t)\mathrm{d}t\in\Psi_{n,z}(D).\]
We have 
\[
(z-T^{(q)}_P)\Td B_{z,D}\equiv S^{(q)}\mod\hat{\mathcal{C}}^{\infty}.\]
We now assume that $b(x,y,z,t)$ is properly supported on $D\times D$. 

Assume that $X=\bigcup^N_{j=1}D_j$, $D_j$ is an open set of $X$ as above, $j=1,\ldots,N$. Let $\chi_j\in\mathcal{C}^\infty_c(D_j)$, $j=1,\ldots,N$, $\sum^N_{j=1}\chi_j\equiv1$ on $X$. Let $\Td\chi_j\in\mathcal{C}^\infty_c(D_j)$, $\Td\chi_j\equiv1$ near ${\rm supp\,}\chi_j$, $j=1,\ldots,N$. Put 
\[B_z:=\sum^N_{j=1}S^{(q)}\Td\chi_j\Td B_{z,D_j}\chi_jS^{(q)}.\]
Then, $B_z\in\Psi_{n,z}(X)$ and $(z-T^{(q)}_P)B_{z}\equiv S^{(q)}\mod\hat{\mathcal{C}}^{\infty}$.  The theorem follows. 
\end{proof}

We need 

\begin{lem}\label{l-gue211205yydII}
For every $s\in\mathbb Z$, there exist $N_s\in\mathbb N$ and $c>0$ such that
\begin{equation}\label{e-gue211205yydIIa}
\norm{(z-T^{(q)}_P)^{-1}S^{(q)}u}_{s-1}\leq\frac{c}{\abs{{\rm Im\,}z}^{N_s}}\norm{u}_s,
\end{equation}
for every $z\in I^{\mathbb C}$ with ${\rm Im\,}z\neq0$ and every $u\in\Omega^{0,q}(X)$.

Thus, $(z-T^{(q)}_P)^{-1}S^{(q)}$ can be extended to a continuous operator: 
\[(z-T^{(q)}_P)^{-1}S^{(q)}: H^s(X,T^{*0,q}X)\To H^{s-1}(X,T^{*0,q}X)\]
and we have the estimates \eqref{e-gue211205yydIIa}, for every $s\in\mathbb Z$, every $z\in I^{\mathbb C}$ with ${\rm Im\,}z\neq0$  and every $u\in H^s(X,T^{*0,q}X)$.
\end{lem}

\begin{proof}
Fix $s\in\mathbb Z$, $s\leq1$. From Theorem~\ref{lem:Tgmf}, we write 
\begin{equation}\label{e-gue211210yyd}
(z-T^{(q)}_P)B_z\equiv S^{(q)}+S^{(q)}R_z,
\end{equation}
where $R_z$ is $z$-smoothing. From \eqref{e-gue211210yyd}, we have 
\begin{equation}\label{e-gue211210yydI}
B_z\equiv (z-T^{(q)}_P)^{-1}S^{(q)}+(z-T^{(q)}_P)^{-1}S^{(q)}R_z.
\end{equation}
Note that 
\begin{equation}\label{e-gue211210yydII}
\norm{(z-T^{(q)}_P)^{-1}u}\leq\frac{1}{\abs{{\rm Im\,}z}}\norm{u},
\end{equation}
for every $z\in I^{\mathbb C}$ with ${\rm Im\,}z\neq0$ and every $u\in L^2_{(0,q)}(X)$. Since $R_z$ is $z$-smoothing, we have 
\begin{equation}\label{e-gue211210yydIII}
\norm{R_zu}\leq\frac{c}{\abs{{\rm Im\,}z}^{N_s}}\norm{u}_s,
\end{equation}
for every $z\in I^{\mathbb C}$ with ${\rm Im\,}z\neq0$ and every $u\in H^s(X,T^{*0,q}X)$, where $c>0$ and $N_s\in\mathbb N$. From \eqref{e-gue211210yydII} and \eqref{e-gue211210yydIII}, we deduce that 
\begin{equation}\label{e-gue211210yyda}
\norm{(z-T^{(q)}_P)^{-1}S^{(q)}R_zu}_{s-1}\leq\norm{(z-T^{(q)}_P)^{-1}S^{(q)}R_zu}\leq\frac{c}{\abs{{\rm Im\,}z}^{N_s}}\norm{u}_s,
\end{equation}
for every $z\in I^{\mathbb C}$ with ${\rm Im\,}z\neq0$ and every $u\in H^s(X,T^{*0,q}X)$, where $c>0$ and $N_s\in\mathbb N$.
From Lemma~\ref{l-gue211205yydI},\eqref{e-gue211210yydI} and \eqref{e-gue211210yyda}, we get \eqref{e-gue211205yydIIa} for $s\leq1$. 
Since $(z-T^{(q)}_P)^{-1}S^{(q)}$ is self-adjoint, by taking adjoint in \eqref{e-gue211205yydIIa}, we get \eqref{e-gue211205yydIIa} for all $s\in\mathbb Z$.
The lemma follows. 
\end{proof}

\begin{lem}\label{l-gue211205yydIIa}
We have 
\begin{equation}\label{e-gue211211yyd}
(z-T^{(q)}_P)^{-1}S^{(q)}\equiv B_z\mod\hat{\mathcal{C}}^\infty, 
\end{equation}
where $B_z$ is as in Theorem~\ref{lem:Tgmf}.
\end{lem}

\begin{proof}
From \eqref{e-gue211210yydI}, we have
\begin{equation}\label{e-gue211211yydm}
B_z\equiv (z-T^{(q)}_P)^{-1}S^{(q)}+(z-T^{(q)}_P)^{-1}S^{(q)}R_z, 
\end{equation}
where $R_z$ is $z$-smoothing. From Lemma~\ref{l-gue211205yydII}, we see that for every $s_1, s_2\in\mathbb Z$, there exist $c>0$ and $N>0$ such that 
\begin{equation}\label{e-gue211211yydI}
\norm{(z-T^{(q)}_P)^{-1}S^{(q)}R_zu}_{s_2}\leq\frac{c}{\abs{{\rm Im\,}z}^{N}}\norm{u}_{s_1},
\end{equation}
for every $z\in I^{\mathbb C}$ with ${\rm Im\,}z\neq0$ and every $u\in\Omega^{0,q}(X)$. From \eqref{e-gue211211yydI} and Sobolev embedding theorem, we get 
\begin{equation}\label{e-gue211211yydII}
(z-T^{(q)}_P)^{-1}S^{(q)}R_z\equiv0\mod\hat{\mathcal{C}}^\infty. 
\end{equation}
From \eqref{e-gue211211yydm} and \eqref{e-gue211211yydII}, the lemma follows. 
\end{proof}

We can  now prove one of the main results of this work

\begin{thm}\label{t-gue211211yyd}
With the notations and assumptions above, assume that $0\notin{\rm supp\,}\tau$. Then, $\tau(T^{(q)}_P)\in\Psi_n(X)$ with 
\begin{equation}\label{e-gue211212yyd}
\mbox{$\sigma^0_{\tau(T^{(q)}_{P})}(x,x)=\frac{1}{2}\pi^{-n-1}\abs{{\rm det\,}\mathcal L_x}\tau(\sigma^0_P(x,-\omega_0(x)))\pi_{x,n_-}$, for every $x\in X$}.
\end{equation}
\end{thm}
 
\begin{proof}
Since $0\notin{\rm supp\,}\tau$, it is not difficult to see that 
\[\tau(T^{(q)}_P)=\tau(T^{(q)}_P)\circ S^{(q)}=S^{(q)}\circ\tau(T^{(q)}_P).\]
From this observation and \eqref{eq:chinh}, we have 
\begin{equation}\label{e-gue211212yyda}
\tau(T^{(q)}_P)=-\frac{1}{\pi}\int_{\mathbb{C}}\frac{\partial \tilde{\tau}(z)}{\partial \bar{z}} (z-T^{(q)}_P)^{-1}\circ S^{(q)} \mathrm{d}x\,\mathrm{d}y.
\end{equation}
From Lemma~\ref{l-gue211205yydIIa} and \eqref{e-gue211212yyda}, we have 
\begin{equation}\label{e-gue211212yydb}
\tau(T^{(q)}_P)\equiv-\frac{1}{\pi}\int_{\mathbb{C}}\frac{\partial \tilde{\tau}(z)}{\partial \bar{z}}B_z \mathrm{d}x\,\mathrm{d}y.
\end{equation}
From \eqref{e-gue21212yydc}, \eqref{e-gue211212yydb} and Cauchy integral formula, we get 
\eqref{e-gue211212yyd}. 
\end{proof} 

If $q=n_-=n_+$, we can repeat the proof of Theorem~\ref{e-gue211212yyd} with minor changes and deduce that 

\begin{thm}\label{t-gue211211yyda}
With the notations and assumptions above, assume that $0\notin{\rm supp\,}\tau$. Suppose $q=n_-=n_+$. Then, $\tau(T^{(q)}_P)\in\Psi_n(X)$ with 
\begin{equation}\label{e-gue211212ycd}
\begin{split}
&\mbox{$\sigma^0_{\tau(T^{(q)}_{P}),-}(x,x)=\frac{1}{2}\pi^{-n-1}\abs{{\rm det\,}\mathcal L_x}\tau(\sigma^0_P(x,-\omega_0(x)))\pi_{x,n_-}$, for every $x\in X$},\\
&\mbox{$\sigma^0_{\tau(T^{(q)}_{P}),+}(x,x)=\frac{1}{2}\pi^{-n-1}\abs{{\rm det\,}\mathcal L_x}\tau(\sigma^0_P(x,\omega_0(x)))\pi_{x,n_+}$, for every $x\in X$}.
\end{split}
\end{equation}
\end{thm}

\subsection{Group actions on CR manifolds}\label{s-gue211213yyd}

We recall that $J:HX\To HX$ is the complex structure map. In this section we assume that $X$ admits an actions of a $d$-dimensional connected compact Lie group $G$. We assume throughout that

\begin{ass}\label{a-gue170123I}
The group $G$ acts via contactomorphisms and holomorphisms. This means that $g^\ast\omega_0=\omega_0$ on $X$ and also $g_\ast J=Jg_\ast$ on $HX$, for every $g\in G$, where $g^*$ and $g_*$ denote the pull-back map and push-forward map of $G$, respectively. 
\end{ass}

We denote the Lie algebra of the group $G$ by $\mathfrak{g}$. Furthermore for every $\xi \in \mathfrak{g}$, we write $\xi_X$ to denote the infinitesimal vector field on $X$ induced by $\xi$, which is given by $$(\xi_X u)(x)=\frac{\partial}{\partial t}\left(u(\exp_G(t\xi)\circ x)\right)|_{t=0}\,,$$ for any smooth function $u\in \mathcal{C}^\infty(X)$. 

\begin{defn}\label{d-gue170124}
	The momentum map associated to the contact form $\omega_0$ is a map $\mu:X \to \mathfrak{g}^*$ such that
	\begin{equation}\label{E:cmpm}
		\langle \mu(x), \xi \rangle = \omega_0(\xi_X(x))
	\end{equation}
	 for all $x \in X$ and $\xi \in \mathfrak{g}$.
\end{defn}

Let us denote by $b$ the non-degenerate bi-linear form  
\begin{equation}\label{E:biform}
	b(\cdot , \cdot) = \mathrm{d}\omega_0(\cdot , J\cdot) \text{ on } HX\,,
\end{equation} 
and  
$$\underline{\mathfrak{g}}^{\perp_b}=\set{v\in HX;\, b(\xi_X,v)=0,\ \ \forall \xi_X\in\underline{\mathfrak{g}}}$$
where $\underline{\mathfrak{g}}={\rm Span\,}(\xi_X;\, \xi\in\mathfrak{g})$.

We need also to assume that 

\begin{ass}\label{a-gue170123II}
$0\in \mathfrak{g}^*$ is a regular value for the moment map $\mu$, the action $G$ on the submanifold $\mu^{-1}(0)$ is free and 
\begin{equation}\label{e-gue200120yydI}
\mbox{$\underline{\mathfrak{g}}_x\cap \underline{\mathfrak{g}}^{\perp_b}_x=\set{0}$ }, 
\end{equation}
at every point $x\in\mu^{-1}(0)$.
\end{ass}

By assumptions \ref{a-gue170123I} and \ref{a-gue170123II} above, $Y:=\mu^{-1}(0)$ is a submanifold of $X$ a co-dimension $d$. Let $HY:=HX\cap TY$. Note that if the Levi form is positive at $Y$, then \eqref{e-gue200120yydI} holds. 
Fix a $G$-invariant smooth Hermitian metric $\langle\, \cdot \,|\, \cdot \,\rangle$ on $TX\otimes \mathbb{C}$ so that $T^{1,0}X$ is orthogonal to $T^{0,1}X$, furthermore $\underline{\mathfrak{g}}$ is orthogonal to $HY\cap JHY$ at every point of $Y$, $\langle \, u \,|\, v \, \rangle$ is real if $u, v$ are real tangent vectors and $\langle\,R\,|\,R\,\rangle=1$.

\subsubsection{Invariant Fourier integral operators of Szeg\H{o} type}\label{s-gue211212yyd}

We first recall some results about invariant Fourier integral operators of Szeg\H{o} type from \cite{gh}. For a given $g\in G$ let $$g^*:\Lambda^r_x( T^*X\otimes \mathbb{C})\To\Lambda^r_{g^{-1}\circ x}(T^*X\otimes \mathbb{C})$$ 
be the pull-back map. Since the action is holomorphic, we have that $g^*:T^{*0,q}_xX\To T^{*0,q}_{g^{-1}\circ x}X$, for all $x\in X.$ Thus, for $u\in\Omega^{0,q}(X)$, we have $g^*u\in\Omega^{0,q}(X)$. Furthermore, put $$\Omega^{0,q}(X)^G:=\set{u\in\Omega^{0,q}(X);\, g^*u=u,\ \ \forall g\in G}.$$
Since the Hermitian metric $\langle\,\cdot\,|\,\cdot\,\rangle$ on $ TX \otimes \mathbb{C}$ is $G$-invariant, the $L^2$ inner product $(\,\cdot\,|\,\cdot\,)$ on $\Omega^{0,q}(X)$ induced by $\langle\,\cdot\,|\,\cdot\,\rangle$ is $G$-invariant. Let $u\in L^2_{(0,q)}(X)$ and $g\in G$, we can also define $g^*u$ in the standard way. Put 
\[L^2_{(0,q)}(X)^G:=\set{u\in L^2_{(0,q)}(X);\, g^*u=u,\ \ \forall g\in G}\]
and 
\[({\rm Ker\,}\Box^q_b)^G:={\rm Ker\,}\Box^q_b\cap L^2_{(0,q)}(X)^G.\]
The $G$-invariant Szeg\H{o} projection is the orthogonal projection 
\[S^{(q)}_G:L^2_{(0,q)}(X)\To ({\rm Ker\,}\Box^q_b)^G\]
with respect to $(\,\cdot\,|\,\cdot\,)$. Let $S^{(q)}_G(x,y)\in\mathcal{D}'(X\times X,T^{*0,q}X\boxtimes(T^{*0,q}X)^*)$ be the distribution kernel of $S^{(q)}_G$. Let 
\[Q_G: L^2_{(0,q)}(X)\To L^2_{(0,q)}(X)^G\]
be the orthogonal projection with respect to $(\,\cdot\,|\,\cdot\,)$. Let
\begin{equation}\label{e-gue210729yyd}
		S^G_{\mp}:=Q_G\circ S_{\mp}: L^2_{(0,q)}(X)\To ({\rm Ker\,}\Box^q_b)^G\,,
\end{equation}
where $S_-$, $S_+$ are as in Theorem~\ref{t-gue161109I}. We have $S^{(q)}_G=S^G_-+S^G_+$. If $q\neq n_+$, then $S^G_+\equiv0$. Recall that we work with the assumption that $q=n_-$. 

\begin{defn}\label{d-gue210729yyd}
	Let $P\in L^{\ell}_{{\rm cl\,}}(X,T^{*0,q}X\boxtimes(T^{*0,q}X)^*)$, $\ell\in\mathbb Z$, $\ell\leq0$. We say that 
	$P$ is in $L^{\ell}_{{\rm cl\,}}(X,T^{*0,q}X\boxtimes(T^{*0,q}X)^*)^G$ if $g^*(Pu)=P(g^*u)$, for every $u\in L^2_{(0,q)}(X)$ and every $g\in G$. 
\end{defn}

\begin{defn}\label{d-gue210801yyd}
	Let $P\in L^\ell_{{\rm cl\,}}(X,T^{*0,q}X\boxtimes(T^{*0,q}X)^*)^G$ with scalar principal symbol, $\ell\leq0$, $\ell\in\mathbb Z$.
	The $G$-invariant Toeplitz operator is given by 
	\begin{equation}\label{e-gue210707ycdz}
			T^{(q)}_{P,G}:=S^{(q)}_G\circ P\circ S^{(q)}_G: L^2_{(0,q)}(X)\To({\rm Ker\,}\Box^q_b)^G,\,
	\end{equation}
	Let $f\in\mathcal{C}^\infty(X)^G$, we write $T^{(q)}_{f,G}:=T^{(q)}_{M_f,G}$.
\end{defn}

We assume that the action is free near $\mu^{-1}(0)$. Fix $x\in\mu^{-1}(0)$, consider the linear map 
\[
	R_x:\underline{\mathfrak{g}}_x\To\underline{\mathfrak{g}}_x,\,
	u\mapsto R_xu,\ \ \langle\,R_xu\,|\,v\,\rangle=\langle\,\mathrm{d}\omega_0(x)\,,\,Ju\wedge v\,\rangle.
\]
Let $\det R_x=\lambda_1(x)\cdots\lambda_d(x)$, where $\lambda_j(x)$, $j=1,2,\ldots,d$, are the eigenvalues of $R_x$. Fix $x\in\mu^{-1}(0)$, put $Y_x=\set{g\circ x;\, g\in G}$. $Y_x$ is a $d$-dimensional submanifold of $X$. The $G$-invariant Hermitian metric $\langle\,\cdot\,|\,\cdot\,\rangle$ induces a volume form $\mathrm{dV}_{Y_x}$ on $Y_x$. Put 
\begin{equation}\label{e-gue170108em}
	V_{{\rm eff\,}}(x):=\int_{Y_x}\mathrm{dV}_{Y_x}.
\end{equation} 

As before, for simplicity, until further notice, we always assume that $n_-\neq n_+$. 

We recall Theorem $4.1$ in \cite{gh}. 

\begin{thm} \label{thm:toeplitz}
	Let $P\in L^\ell_{{\rm cl\,}}(X,T^{*0,q}X\boxtimes(T^{*0,q}X)^*)^G$ with scalar principal symbol, $\ell\leq0$, $\ell\in\mathbb Z$.
	Let $D$ be an open set in $X$ such that the intersection $\mu^{-1}(0)\cap D= \emptyset$. Then $T^{(q)}_{P,G}\equiv 0$ on $D$.
	
	Let $p\in \mu^{-1}(0)$ and let $U$ be a local neighborhood of $p$ with local coordinates $(x_1,\dots\,x_{2n+1})$. Then, the distributional kernel of $T_{P,G}^{(q)}$ satisfies
	\begin{equation}\label{e-gue210802yyd}
		T_{P,G}^{(q)}(x,y)\equiv\int_0^{\infty} e^{\imath t\,\Phi_-( x,y)}a_{P}( x,\,y,\,t)\,\mathrm{d}t\ \ \ \mbox{on $U\times U$}.
	\end{equation}
	The phase $\Phi_-$ is described in local coordinates below and it is equal to the phase of $S_{-}^{G}(x,\,y)$ in~\cite{hsiaohuang}, the symbol 
	$a_{P}$ satisfies the following properties
	\[a_{P}(x,\,y,\,t)\sim \sum_{j=0}^{+\infty} a_{P}^{j}(x,\,y)\,t^{n-d/2-j}\]
	in $S^{n-d/2}_{1,0}(U\times U\times \mathbb{R},\,T^{*\,0,q}X\boxtimes(T^{*\,0,q}X)^*)$, 
	\[a_{P}^{j}(x,\,y)\in\mathcal{C}^{\infty}(U\times U,T^{*0,q}X\boxtimes(T^{*0,q}X)^*),\ \ j\in\mathbb N_0,\]
	and for every $x\in\mu^{-1}(0)$, 
	\begin{equation}\label{e-gue210802yydIm}
		a_{P}^{0}(x,\,x)=2^{d-1}\frac{1}{V_{{\rm eff\,}}(x)}\pi^{-n-1+\frac{d}{2}}\abs{\det R_x^{-\frac{1}{2}}}\abs{\det\mathcal{L}_{x}}\sigma^0_P(x,-\omega_0(x))\pi_{x,n_-}.	\end{equation}
\end{thm}

We do make use of local coordinates defined in \cite{hsiaohuang} to write the phase in local coordinates, we will not recall it here explicitly, we refer to Theorem 3.7 in \cite{hsiaohuang}.
For convenience of the reader we just mention that the phase function $\Phi_-(x,y)\in\mathcal{C}^\infty(U\times U)$ is $G$-invariant which results in the fact that it is independent of $(x_1,\ldots,x_d)$ and $(y_1,\ldots,y_d)$. Hence, we have that $\Phi_-(x,y)=\Phi_-((0,x''),(0,y'')):=\Phi_-(x'',y'')$, $x''=(x_{d+1},\ldots,x_{2n+1})$,  $y''=(y_{d+1},\ldots,y_{2n+1})$. Moreover, there is a constant $c>0$ such that 
\begin{equation}\label{e-gue170126}
	{\rm Im\,}\Phi_-(x'',y'')\geq c\Bigr(\abs{\hat x''}^2+\abs{\hat y''}^2+\abs{\mathring{x}''-\mathring{y}''}^2\Bigr),\ \ \forall ((0,x''),(0,y''))\in U\times U,
\end{equation}
where $\hat x''=(x_{d+1},\ldots,x_{2d})$, $\hat y''=(y_{d+1},\ldots,y_{2d})$, $\mathring{x}''=(x_{d+1},\ldots,x_{2n})$, $\mathring{y}''=(y_{d+1},\ldots,y_{2n})$.

Let $P\in L^{0}_{{\rm cl\,}}(X,T^{*0,q}X\boxtimes(T^{*0,q}X)^*)^G$ be a $G$-invariant self-adjoint classical pseudodifferential operator on $X$ with scalar principal symbol. Then, it is not difficult to see that the Toepliotz operator $T^{(q)}_{P,G}$ is self-adjoint.  Let $\tau\in\mathcal{C}^\infty_c(\mathbb R)$. We have 
\begin{equation}\label{e-gue211212yydm}
\tau(T^{(q)}_{P,G})=\tau(T^{(q)}_{P,G})\circ Q_G.
\end{equation}
From Theorem~\ref{t-gue211211yyd}, \eqref{e-gue211212yyd} and \eqref{e-gue211212yydm}, we can repeat the proof of Theorem $4.1$ in \cite{gh} and get 

\begin{thm} \label{thm:toeplitz-m}
Recall that we work with the assumption that $n_-\neq n_+$. 
With the notations and assumptions above, let $\tau\in\mathcal{C}^\infty_c(\mathbb R)$ with $0\notin{\rm supp\,}\tau$. Let $P\in L^0_{{\rm cl\,}}(X,T^{*0,q}X\boxtimes(T^{*0,q}X)^*)^G$ with scalar principal symbol. 
	Let $D$ be an open set in $X$ such that the intersection $\mu^{-1}(0)\cap D= \emptyset$. Then $\tau(T^{(q)}_{P,G})\equiv 0$ on $D$.
	
	Let $p\in \mu^{-1}(0)$ and let $U$ be a local neighborhood of $p$ with local coordinates $(x_1,\dots\,x_{2n+1})$. Then, the distributional kernel of $\tau(T_{P,G}^{(q)})$ satisfies
	\begin{equation}\label{e-gue210802yydam}
	\tau(T_{P,G}^{(q)})(x,y)\equiv\int_0^{\infty} e^{\imath t\,\Phi_-( x,y)}a_{\tau}( x,\,y,\,t)\,\mathrm{d}t\ \ \ \mbox{on $U\times U$}.
	\end{equation}
	The phase $\Phi_-$ is as in Theorem~\ref{thm:toeplitz}, 
	$a_{\tau}$ satisfies the following properties
	\[a_{\tau}(x,\,y,\,t)\sim \sum_{j=0}^{+\infty} a_{\tau}^{j}(x,\,y)\,t^{n-d/2-j}\]
	in $S^{n-d/2}_{1,0}(U\times U\times \mathbb{R},\,T^{*\,0,q}X\boxtimes(T^{*\,0,q}X)^*)$, 
	\[a_{\tau}^{j}(x,\,y)\in\mathcal{C}^{\infty}(U\times U,T^{*0,q}X\boxtimes(T^{*0,q}X)^*),\ \ j\in\mathbb N_0,\]
	and for every $x\in\mu^{-1}(0)$, 
	\begin{equation}\label{e-gue210802yydIn}
		a_{\tau}^{0}(x,\,x)=2^{d-1}\frac{1}{V_{{\rm eff\,}}(x)}\pi^{-n-1+\frac{d}{2}}\abs{\det R_x^{-\frac{1}{2}}}\abs{\det\mathcal{L}_{x}}\tau(\sigma^0_P(x,-\omega_0(x)))\pi_{x,n_-}.	\end{equation}
\end{thm}

Similarly, if $n_-=n_+$, we have 

\begin{thm} \label{thm:toeplitz-n}
With the notations and assumptions above, assume that $n_-=n_+$. Let $\tau\in\mathcal{C}^\infty_c(\mathbb R)$ with $0\notin{\rm supp\,}\tau$. Let $P\in L^0_{{\rm cl\,}}(X,T^{*0,q}X\boxtimes(T^{*0,q}X)^*)^G$ with scalar principal symbol. 
	Let $D$ be an open set in $X$ such that the intersection $\mu^{-1}(0)\cap D= \emptyset$. Then $\tau(T^{(q)}_{P,G})\equiv 0$ on $D$.
	
	Let $p\in \mu^{-1}(0)$ and let $U$ be a local neighborhood of $p$ with local coordinates $(x_1,\dots\,x_{2n+1})$. Then, the distributional kernel of $\tau(T_{P,G}^{(q)})$ satisfies
	\begin{equation}\label{e-gue210802yyda}
	\tau(T_{P,G}^{(q)})(x,y)\equiv\int_0^{\infty} e^{\imath t\,\Phi_-( x,y)}a_{\tau,-}( x,\,y,\,t)\mathrm{d}t+\int_0^{\infty} e^{\imath t\,\Phi_+( x,y)}a_{\tau,+}( x,\,y,\,t\,)\mathrm{d}t\ \ \ \mbox{on $U\times U$}.
	\end{equation}
	The phases $\Phi_{\mp}$ are equal to the phases of $S_{\mp}^{G}(x,\,y)$ in~\cite{hsiaohuang}, 
	$a_{\tau,\mp}$ satisfies the following properties
	\[a_{\tau,\,p}(x,\,y,\,t)\sim \sum_{j=0}^{+\infty} a_{\tau,\mp}^{j}(x,\,y)\,t^{n-d/2-j}\]
	in $S^{n-d/2}_{1,0}(U\times U\times \mathbb{R},\,T^{*\,0,q}X\boxtimes(T^{*\,0,q}X)^*)$, 
	\[a_{\tau,\mp}^{j}(x,\,y)\in\mathcal{C}^{\infty}(U\times U,T^{*0,q}X\boxtimes(T^{*0,q}X)^*),\ \ j\in\mathbb N_0,\]
	and for every $x\in\mu^{-1}(0)$, 
	\begin{equation}\label{e-gue210802yydI}\begin{split}
		&a_{\tau,-}^{0}(x,\,x)=2^{d-1}\frac{1}{V_{{\rm eff\,}}(x)}\pi^{-n-1+\frac{d}{2}}\abs{\det R_x^{-\frac{1}{2}}}\abs{\det\mathcal{L}_{x}}\tau(\sigma^0_P(x,-\omega_0(x)))\pi_{x,n_-},\\
		&a_{\tau,+}^{0}(x,\,x)=2^{d-1}\frac{1}{V_{{\rm eff\,}}(x)}\pi^{-n-1+\frac{d}{2}}\abs{\det R_x^{-\frac{1}{2}}}\abs{\det\mathcal{L}_{x}}\tau(\sigma^0_P(x,\omega_0(x)))\pi_{x,n_+}.
		\end{split}
			\end{equation}
\end{thm}

\section{Quantization commutes with reduction for Toeplitz operators}\label{s-gue211213yyds}

We will use the same notations and assumptions as Section~\ref{s-gue211213yyd}. Let $X_G:=\mu^{-1}(0)/G$. In \cite{hsiaohuang}, it was proved that $X_{G}$ is a CR manifold with natural CR structure induced by $T^{1,0}X$ of dimension $2n - 2d + 1$.
 Let $P\in L^0_{{\rm cl\,}}(X,T^{*0,q}X\boxtimes(T^{*0,q}X)^*)^G$ be a scalar pseudodifferential operator. The operator $P$ gives rise to ${P}_{X_G}$ on the CR reduction $X_G=\mu^{-1}({0})/G$. The Hermitian metric $\langle\,\cdot\,|\,\cdot\,\rangle$ on $TX\otimes\mathbb{C}$ induces an Hermitian metric $\langle\,\cdot\,|\,\cdot\,\rangle_{X_G}$ on $TX_G\otimes\mathbb{C}$ and let $(\,\cdot\,|\,\cdot\,)_{X_G}$ be the $L^2$ inner product on $L^2_{(0,q)}(X_G)$ induced by $\langle\,\cdot\,|\,\cdot\,\rangle_{X_G}$. We write $\mathrm{dV}_{X_G}$ to denote the volume form on $X_G$ induced by $\langle\,\cdot\,|\,\cdot\,\rangle_{X_G}$. Let $\Box ^{q}_{b,X_G}$ be the Kohn Laplacian for $(0,q)$ forms on $X_G$ defined with respect to $\langle\,\cdot\,|\,\cdot\,\rangle_{X_G}$ and $(\,\cdot\,|\,\cdot\,)_{X_G}$. 
 We can define the standard Toeplitz operator associated to ${P}_{X_G}$:
\[T^{(q)}_{P_{X_G}}:= S^{(q)}_{X_G}\circ {P}_{X_G} \circ S^{(q)}_{X_G}:L^2_{(0,q)}(X_G)\To L^2_{(0,q)}(X_G),\]
where $S^{(q)}_{X_G}:L^2_{(0,q)}(X_G)\rightarrow \mathrm{Ker}\,\Box ^{q}_{b,X_G}$ is the Szeg\H{o} projection on $X_G$.  

Let $\mathcal{L}_{X_{G}}$ be the Levi form on $X_{G}$ induced naturally from the Levi form $\mathcal{L}$ on $X$. Since $\mathcal{L}$ is non-degenerate of constant signature, $\mathcal{L}_{X_{G}}$ is also non-degenerate of constant signature. We denote by $n_{X_G,-}$ (respectively $n_{X_G,+}$) the number of negative (respectively positive) eigenvalues of $\mathcal{L}_{X_G}$. 
For simplicity, until further notice, we assume that $n_{X_G,-}\neq n_{X_G,+}$. We introduce the map $\sigma$. 
Since $\underline{\mathfrak{g}}_x$ is orthogonal to $H_xY\cap JH_xY$ and $H_xY\cap JH_xY\subset\underline{\mathfrak{g}}^{\perp_b}_x$, we can find a $G$-invariant orthonormal basis $\set{Z_1,\ldots,Z_n}$ of $T^{1,0}X$ on $Y$ such that 
\[
\renewcommand{\arraystretch}{1.2}
\begin{array}{ll}
&\mathcal{L}_x(Z_j(x),\ol Z_k(x))=\delta_{j,k}\lambda_j(x),\ \ j,k=1,\ldots,n\,,\\
&Z_j(x)\in\underline{\mathfrak{g}}_x+iJ\underline{\mathfrak{g}}_x,\ \ j=1,2,\ldots,d
\end{array}
\]
and
\[Z_j(x)\in H_xY\otimes \mathbb{C}\cap J( H_xY\otimes \mathbb{C}),\ \ j=d+1,\ldots,n\,.
\]
Let $\set{e_1,\ldots,e_n}$ denote the orthonormal basis of $T^{*0,1}X$ on $Y$, dual to $\set{\ol Z_1,\ldots,\ol Z_n}$. Fix $s=0,1,2,\ldots,n-d$. For $x\in Y$, put 
\[
B^{*0,s}_xX=\left\{ \sum_{d+1\leq j_1<\cdots<j_s\leq n}a_{j_1,\ldots,j_s}e_{j_1}\wedge\cdots\wedge e_{j_s};\, a_{j_1,\ldots,j_s}\in\Complex,\ \forall d+1\leq j_1<\cdots<j_s\leq n \right\}
\]
and let $B^{*0,s}X$ be the vector bundle of $Y$ with fiber $B^{*0,s}_xX$, $x\in Y$. Let $\mathcal{C}^\infty(Y,B^{*0,s}X)^G$ denote the set of all $G$-invariant sections of $Y$ with values in $B^{*0,s}X$. Let 
\[
\iota_G:\mathcal{C}^\infty(Y,B^{*0,s}X)^G\To\Omega^{0,s}(Y_G)
\]
be the natural identification. 

Assume that $\lambda_1<0,\ldots,\lambda_r<0$, and $\lambda_{d+1}<0,\ldots,\lambda_{n_--r+d}<0$, where $n_--r=n_{X_G,-}$. For $x\in Y$, put 
\[
\hat{\mathcal{N}}(x,n_-)=\set{ce_{d+1}\wedge\cdots\wedge e_{n_--r+d};\, c\in\mathbb C},
\]
and let 
\[
\renewcommand{\arraystretch}{1.2}
\begin{array}{c}
\hat p_{-}=\hat p_{x,-}:\mathcal{N}(x,n_-)\To \hat{\mathcal{N}}(x,n_-),\\
u=ce_1\wedge\cdots\wedge e_r\wedge e_{d+1}\wedge\cdots\wedge e_{n_--r+d}\To ce_{d+1}\wedge\cdots\wedge e_{n_--r+d}.
\end{array}
\]
We can define $\hat p_+$ in the similar way. 
Let $\iota:Y\To X$ be the natural inclusion and let $\iota^*:\Omega^{0,q}(X)\To\Omega^{0,q}(Y)$ be the pull-back of $\iota$. 
Let 
\[
f(x)=\sqrt{V_{{\rm eff\,}}(x)}\abs{\det\,R_x}^{-\frac{1}{4}}\in\mathcal{C}^\infty(Y)^G.
\]
Let $E\in L^{-\frac{d}{4}}_{{\rm cl\,}}(X_G,T^{*0,n_{X_G,-}}X_G\boxtimes(T^{*0,n_{X_G,-}}X_G)^*)$ be a scalar elliptic pseudodifferential operator with principal symbol 
$\sigma^0_E(x,\xi)=\abs{\xi}^{-\frac{d}{4}}$. Let $\tau\in\mathcal{C}^\infty_c(\mathbb R)$. 
Let
\begin{equation}\label{e-gue211213yyda}
\begin{split}
\sigma: E_\tau(T^{(n_-)}_{P,G})\cap\Omega^{0,n_-}(X)&\To E_\tau(T^{(n_{X_G,-})}_{P_{X_G}})\cap\Omega^{0,n_{X_G,-}}(X_G),\\
u&\To  \tau(T^{(n_{X_G,-})}_{P_{X_G}})\circ E\circ\iota_G\circ \hat p_-\circ\pi_{x,n_-}\circ f\circ \iota^*\circ u.
\end{split}\end{equation}
We can extend $\sigma$ to $\Omega^{0,n_-}(X)$ by 
\begin{equation}\label{e-gue180308II}
\begin{split}
\sigma: \Omega^{0,n_-}(X)&\To E_\tau(T^{(n_{X_G,-})}_{P_{X_G}})
\cap\Omega^{0,n_{X_G,-}}(X_G),\\
u&\To\tau(T^{(n_{X_G,-})}_{P_{X_G}})\circ E\circ\iota_G\circ \hat p_-\circ\pi_{x,n_-}\circ f\circ \iota^*\circ\tau(T^{(n_-)}_{P,G})u.\,.
\end{split}
\end{equation}
Let $\sigma^{\,*}: \Omega^{0,n_{X_G,-}}(X_{G})\To \mathcal{D}'(X,T^{*0,n_-}X)$ be
the formal adjoint of $\sigma$. From Theorem~\ref{t-gue211211yyd} and Theorem~\ref{thm:toeplitz-m}, we can repeat the method in Section 4.2 of~\cite{hmm} and deduce 
that

\begin{thm}\label{t-gue211213yyd}
With the same notations and assumptions above, suppose that $\Box^{n_{X_G,-}}_b$ has closed range. Let $\tau\in\mathcal{C}^\infty_c(\mathbb R)$ with $0\notin{\rm supp\,}\tau$. Then, $\sigma^{\,*}$ maps $\Omega^{0,n_{X_G,-}}(X_{G})$ continuously to $\Omega^{0,n_-}(X)$. Hence, $\sigma^*\,\sigma$ is a well-defined continuous operator 
\[\sigma^*\,\sigma: \Omega^{0,n_-}(X)\To\Omega^{0,n_-}(X).\]

Furthermore, the distribution kernel of $\sigma^*\,\sigma$ satisfies 
\begin{equation}\label{e-gue211213yydp}
\sigma^*\,\sigma=\tau(T^{(q)}_{P,G})^2+\Gamma,
\end{equation}
where $\Gamma: \Omega^{0,n_-}(X)\To\Omega^{0,n_-}(X)$ is a continuous map and the distribution kernel of $\Gamma$ satisfies the following: Let $D$ be an open set in $X$ such that the intersection $\mu^{-1}(0)\cap D= \emptyset$. Then $\Gamma\equiv 0$ on $D$. Let $p\in \mu^{-1}(0)$ and let $U$ be a local neighborhood of $p$ with local coordinates $(x_1,\dots\,x_{2n+1})$. Then, the distributional kernel of $\Gamma$ satisfies
	\begin{equation}\label{e-gue211213yydq}
	\Gamma(x,y)\equiv\int_0^{\infty} e^{\imath t\,\Phi_-( x,y)}r( x,\,y,\,t)\,\mathrm{d}t\ \ \ \mbox{on $U\times U$}.
	\end{equation}
	The phase $\Phi_-$ is as in Theorem~\ref{thm:toeplitz}, 
	$r$ satisfies the following properties
	\[r(x,\,y,\,t)\sim \sum_{j=0}^{+\infty} r_j(x,\,y)\,t^{n-d/2-j}\]
	in $S^{n-d/2}_{1,0}(U\times U\times \mathbb{R},\,T^{*\,0,q}X\boxtimes(T^{*\,0,q}X)^*)$, 
	\[r_j(x,\,y)\in\mathcal{C}^{\infty}(U\times U,T^{*0,q}X\boxtimes(T^{*0,q}X)^*),\ \ j\in\mathbb N_0,\]
	and for every $x\in\mu^{-1}(0)$, $r_0(x,x)=0$. 
\end{thm} 

Since the leading term of $\Gamma$ vanishes, we can repeat the method in  Section 4.1 of~\cite{hmm} and deduce 
that

\begin{lem}\label{l-gue211213ycd}
With the same notations and assumptions above, suppose that $\Box^{n_{X_G,-}}_b$ has closed range. Let $\Gamma$ be as in \eqref{e-gue211213yydp}. There is a $\varepsilon>0$ such that $\Gamma$ is continuous: 
\[\Gamma: H^s(X,T^{*0,n_-}X)\To H^{s+\varepsilon}(X,T^{*0,n_-}X),\]
for every $s\in\mathbb Z$. 
\end{lem}

From \eqref{e-gue211213yydp} and Lemma~\ref{l-gue211213ycd}, we see that $\sigma$ can be extended to $E_\tau(T^{(n_-)}_{P,G})$ by density and thus, 
\begin{equation}\label{e-gue211213yydr}
\sigma: E_\tau(T^{(n_-)}_{P,G})\To E_\tau(T^{(n_{X_G,-})}_{P_{X_G}}).
\end{equation} 

We can now prove

\begin{thm}\label{t-gue211213yydI}
With the same notations and assumptions above, suppose that $\Box^{n_{X_G,-}}_b$ has closed range. Let $\tau\in\mathcal{C}^\infty_c(\mathbb R)$ with $0\notin{\rm supp\,}\tau$.
The map $\sigma$  given by \eqref{e-gue180308II} and \eqref{e-gue211213yydr} has the following properties: (i) For every open bounded interval $I$, $0\notin\overline{I}$,  with $\tau\equiv1$ on $I$, we have that ${\rm Ker\,}\sigma\cap E_{I}(T^{(n_-)}_{P,G})$ is a finite dimensional subspace of $\Omega^{0,n_-}(X)$. (ii) For every open bounded interval $I$, $0\notin\overline{I}$, with $\tau\equiv1$ on $I$, we have that ${\rm Coker\,}\sigma\cap E_{I}(T^{(n_{X_G,-})}_{P_{X_G}})$ is a finite dimensional subspace of $\Omega^{0,n_{X_G,-}}(X_G)$.
\end{thm}

\begin{proof}
Fix an open bounded interval $I$, $0\notin\overline{I}$ with $\tau\equiv1$ on $I$. Let $u\in{\rm Ker\,}\sigma\cap E_{I}(T^{(n_-)}_{P,G})$. From \eqref{e-gue211213yydp}, we have 
\begin{equation}\label{e-gue211213ycdt}
0=\sigma^*\,\sigma u=(\tau(T^{(q)}_{P,G})^2+\Gamma)u=u+\Gamma u.
\end{equation}
From Lemma~\ref{l-gue211213ycd} and \eqref{e-gue211213ycdt}, we deduce that $u\in H^{\varepsilon}(X,T^{*0,n_-}X)$. By using Lemma~\ref{l-gue211213ycd} and \eqref{e-gue211213ycdt} again, we get $u\in H^{2\varepsilon}(X,T^{*0,n_-}X)$. Continuing in this way, we conclude that $u\in\Omega^{0,n_-}(X)$. 

Suppose that ${\rm Ker\,}\sigma\cap E_{I}(T^{(n_-)}_{P,G})$ is not a finite dimensional subspace of $\Omega^{0,n_-}(X)$. We can find $u_j\in{\rm Ker\,}\sigma\cap E_{I}(T^{(n_-)}_{P,G})$, $j=1,2,\ldots$, $(\,u_j\,|\,u_k\,)=\delta_{j,k}$, for every $j, k=1,\ldots$. From Lemma~\ref{l-gue211213ycd} and \eqref{e-gue211213ycdt}, 
we see that $\set{u_j}^{+\infty}_{j=1}$ is a bounded set in $H^{\varepsilon}(X,T^{*0,n_-}X)$. By Rellich's lemma, there is a subsequence $1\leq j_1<j_2<\cdots$, $\lim_{s\To+\infty}j_s=+\infty$, such that $u_{j_s}\To u$ in $L^2_{(0,n_-)}(X)$ as $s\To+\infty$, for some $u\in L^2_{(0,n_-)}(X)$. Since  
$(\,u_j\,|\,u_k\,)=\delta_{j,k}$, for every $j, k=1,\ldots$, we get a contradiction. Thus, ${\rm Ker\,}\sigma\cap E_{I}(T^{(n_-)}_{P,G})$ is a finite dimensional subspace of $\Omega^{0,n_-}(X)$. 

We can repeat the procedure above with minor change and deduce that $${\rm Coker\,}\sigma\cap E_{I}(T^{(n_{X_G,-})}_{P_{X_G}})$$ is a finite dimensional subspace of $\Omega^{0,n_{X_G,-}}(X_G)$.
\end{proof} 

Let $H^{q}_b(X)^G={\rm Ker\,}\Box^q_b\cap L^2_{(0,q)}(X)^G$, $H^{q}_b(X_G):={\rm Ker\,}\Box^q_{b,X_G}$ and take $\tau(x)=1$ near $x=1$. 
As a corollary of Theorem~\ref{t-gue211213yydI}, we deduce 

\begin{cor}\label{c-gue211213yyd}
With the same notations and assumptions above, suppose that $\Box^{n_{X_G,-}}_b$ has closed range. We have 
\[\sigma: H^{n_-}_b(X)^G\To H^{n_{X_G,-}}(X_G)\]
is a Fredholm, that is, ${\rm Ker\,}\sigma$ and ${\rm Coker\,}\sigma$ are finite dimensional subspaces of $\Omega^{0,n_-}(X_G)$ and $\Omega^{0,n_{X_G,-}}(X_G)$ respectively. 
\end{cor}

Our method works well if $n_-=n_+$ or $n_{X_G,-}=n_{X_G,+}$. If $n_-=n_+$, $n_{X_G,-}=n_{X_G,+}$, we define $\sigma$ as \eqref{e-gue180308II} and \eqref{e-gue211213yydr}. 
If $n_-=n_+$, $n_{X_G,-}\neq n_{X_G,+}$, we define 
\begin{equation}\label{e-gue180308IIa}
\begin{split}
\sigma: E_\tau(T^{(n_-)}_{P,G})\cap\Omega^{0,n_-}(X)&\To E_\tau(T^{(n_{X_G,-})}_{P_{X_G}})\cap\Omega^{0,n_{X_G,-}}(X_G)\oplus  E_\tau(T^{(n_{X_G,+})}_{P_{X_G}})\cap\Omega^{0,n_{X_G,-}}(X_G),\\
u&\To v_1\oplus v_2,
\end{split}
\end{equation}
where
\begin{align*}
	v_1:&=\tau(T^{(n_{X_G,-})}_{P_{X_G}})\circ E\circ\iota_G\circ \hat p\circ\pi_{x,n_-}\circ f\circ \iota^*\circ u,\\
	v_2:&=\tau(T^{(n_{X_G,+})}_{P_{X_G}})\circ E\circ\iota_G\circ \hat p_+\circ\pi_{x,n_+}\circ f\circ \iota^*\circ u\,
\end{align*}
and $\sigma$ can be extended to $E_\tau(T^{(n_-)}_{P,G})$ by density and thus, 
\[\sigma: E_\tau(T^{(n_-)}_{P,G})\To E_\tau(T^{(n_{X_G,-})}_{P_{X_G}})\oplus  E_\tau(T^{(n_{X_G,+})}_{P_{X_G}}).\]

If $n_-\neq n_+$, $n_{X_G,-}=n_{X_G,+}$, we define 
\begin{equation}\label{e-gue211214yyd}
\begin{split}
\sigma: E_\tau(T^{(n_-)}_{P,G})\cap\Omega^{0,n_-}(X)\oplus E_\tau(T^{(n_+)}_{P,G})\cap\Omega^{0,n_+}(X)&\To E_\tau(T^{(n_{X_G,-})}_{P_{X_G}})\cap\Omega^{0,n_{X_G,-}}(X_G),\\
u_1\oplus u_2&\mapsto v=v_1+v_2
\end{split}
\end{equation}
where
\begin{align*}&v_1:=\tau(T^{(n_{X_G,-})}_{P_{X_G}})\circ E\circ\iota_G\circ \hat p\circ\pi_{x,n_-}\circ f\circ \iota^*\circ u_1,\\
&v_2:=\tau(T^{(n_{X_G,+})}_{P_{X_G}})\circ E\circ\iota_G\circ \hat p_+\circ\pi_{x,n_+}\circ f\circ \iota^*\circ u_2\,
\end{align*}
and $\sigma$ can be extended to $E_\tau(T^{(n_-)}_{P,G})\oplus E_\tau(T^{(n_+)}_{P,G})$ by density and thus, 
\[\sigma: E_\tau(T^{(n_-)}_{P,G})\oplus E_\tau(T^{(n_+)}_{P,G})\To E_\tau(T^{(n_{X_G,-})}_{P_{X_G}}).\]

Repeating the proof of Theorem~\ref{t-gue211213yyd} with minor changes, we deduce 

\begin{thm}\label{t-gue211214yyd}
With the same notations and assumptions above, let $\tau\in\mathcal{C}^\infty_c(\mathbb R)$ with $0\notin{\rm supp\,}\tau$.
If $n_-=n_+$ and $n_{X_G,-}=n_{X_G,+}$. Suppose that $\Box^{n_{X_G,-}}_b$ has closed range. 
The map 
\[\sigma: E_\tau(T^{(n_-)}_{P,G})\To E_\tau(T^{(n_{X_G,-})}_{P_{X_G}})\]
given by \eqref{e-gue180308II} and \eqref{e-gue211213yydr} has the following properties: \begin{itemize}
	\item[(i)] For every open bounded interval $I$, $0\notin\overline{I}$, with $\tau\equiv1$ on $I$, we have that $${\rm Ker\,}\sigma\cap E_{I}(T^{(n_-)}_{P,G})$$ is a finite dimensional subspace of $\Omega^{0,n_-}(X)$;
	\item[(ii)] for every open bounded interval $I$, $0\notin\overline{I}$, with $\tau\equiv1$ on $I$, we have that $${\rm Coker\,}\sigma\cap E_{I}(T^{(n_{X_G,-})}_{P_{X_G}})$$ is a finite dimensional subspace of $\Omega^{0,n_{X_G,-}}(X_G)$. 
\end{itemize}

If $n_-=n_+$ and $n_{X_G,-}\neq n_{X_G,+}$. Suppose that $\Box^{n_{X_G,-}}_b$ and $\Box^{n_{X_G,+}}_b$ have closed range. 
The map 
\[\sigma: E_\tau(T^{(n_-)}_{P,G})\To E_\tau(T^{(n_{X_G,-})}_{P_{X_G}})\oplus E_\tau(T^{(n_{X_G,+})}_{P_{X_G}})\]
given by \eqref{e-gue180308IIa} has the following properties: 
\begin{itemize}
	\item[(i)] For every open bounded interval $I$, $0\notin\overline{I}$, with $\tau\equiv1$ on $I$, we have that $${\rm Ker\,}\sigma\,\cap E_{I}(T^{(n_-)}_{P,G})$$ is a finite dimensional subspace of $\Omega^{0,n_-}(X)$;
	\item[(ii)] For every open bounded interval $I$, $0\notin\overline{I}$, with $\tau\equiv1$ on $I$, we have that $${\rm Coker\,}\sigma\,\cap\Bigr(E_{I}(T^{(n_{X_G,-})}_{P_{X_G}})\oplus  E_{I}(T^{(n_{X_G,+})}_{P_{X_G}}\Bigr)$$ is a finite dimensional subspace of $\Omega^{0,n_{X_G,-}}(X_G)\oplus \Omega^{0,n_{X_G,+}}(X_G)$.
\end{itemize} 

If $n_-\neq n_+$ and $n_{X_G,-}=n_{X_G,+}$. Suppose that $\Box^{n_{+}}_b$, $\Box^{n_{X_G,-}}_b$ have closed range. Recall that we always assume that $\Box^{n_{-}}_b$ has closed range.
The map 
\[\sigma: E_\tau(T^{(n_-)}_{P,G})\oplus E_\tau(T^{(n_+)}_{P,G})\To E_\tau(T^{(n_{X_G,-})}_{P_{X_G}})\]
given by \eqref{e-gue211214yyd} has the following properties: 
\begin{itemize}
	\item[(i)] For every open bounded interval $I$, $0\notin\overline{I}$, with $\tau\equiv1$ on $I$, we have that $${\rm Ker\,}\sigma\,\cap \Bigr(E_{I}(T^{(n_-)}_{P,G})\oplus E_{I}(T^{(n_+)}_{P,G})\Bigr)$$ is a finite dimensional subspace of $\Omega^{0,n_-}(X)\oplus\Omega^{0,n_+}(X)$. 
	\item[(ii)] For every open bounded interval $I$,  $0\notin\overline{I}$, with $\tau\equiv1$ on $I$, we have that $${\rm Coker\,}\sigma\,\cap E_{I}(T^{(n_{X_G,-})}_{P_{X_G}})$$ is a finite dimensional subspace of $\Omega^{0,n_{X_G,-}}(X_G)$.
\end{itemize}
\end{thm} 

From Lemma~\ref{l-gue211126yydq}, we see that if $\sigma^0_P(x,\mp\omega_0(x))\neq0$ for every $x\in X$, then $\mathrm{Spec}(T^{(q)}_P)\subset I\cup\set{0}$, for some open interval $I$ with $0\notin\overline{I}$. From this observation and Theorem~\ref{t-gue211214yyd}, we get 

\begin{thm}\label{t-gue211214yydI}
With the same notations and assumptions above, assume that $\sigma^0_P(x,\mp\omega_0(x))>0$ for every $x\in X$.  Thus, $\mathrm{Spec}(T^{(n_{\mp})}_{P,G})\subset I\cup\set{0}$  if $\Box^{n_{\mp}}_b$ has closed range, $\mathrm{Spec}(T^{(n_{X_G,\mp})}_{P_{X_G}})\subset I\cup\set{0}$, if $\Box^{n_{X_G,\mp}}_{b,X_G}$ has closed range, for some open bounded interval $I$ with $0\notin\overline{I}$.  Let $\tau\in\mathcal{C}^\infty_c(\mathbb R)$ with $0\notin{\rm supp\,}\tau$ and $\tau\equiv 1$ near $I$. 

If $n_-=n_+$ and $n_{X_G,-}=n_{X_G,+}$ or $n_-\neq n_+$ and $n_{X_G,-}\neq n_{X_G,+}$. Suppose that $\Box^{n_{X_G,-}}_b$ has closed range. 
The map 
\[\sigma: E_{I}(T^{(n_-)}_{P,G})\To E_{I}(T^{(n_{X_G,-})}_{P_{X_G}})\]
given by \eqref{e-gue180308II} and \eqref{e-gue211213yydr} is Fredholm. 

If $n_-=n_+$ and $n_{X_G,-}\neq n_{X_G,+}$. Suppose that $\Box^{n_{X_G,-}}_b$ and $\Box^{n_{X_G,+}}_b$ have closed range. 
The map 
\[\sigma: E_{I}(T^{(n_-)}_{P,G})\To E_{I}(T^{(n_{X_G,-})}_{P_{X_G}})\oplus E_{I}(T^{(n_{X_G,+})}_{P_{X_G}})\]
given by \eqref{e-gue180308IIa} is Fredholm. 

If $n_-\neq n_+$ and $n_{X_G,-}=n_{X_G,+}$. Suppose that $\Box^{n_{+}}_b$ and $\Box^{n_{X_G,-}}_b$ have closed range. Recall that we always assume that $\Box^{n_{-}}_b$ has closed range.
The map 
\[\sigma: E_{I}(T^{(n_-)}_{P,G})\oplus E_{I}(T^{(n_+)}_{P,G})\To E_{I}(T^{(n_{X_G,-})}_{P_{X_G}})\]
given by \eqref{e-gue211214yyd} is Fredholm. 
\end{thm} 





\section{CR manifolds with circle action}\label{s-gue211214yyd}

\subsection{Some standard notations in semi-classical analysis}\label{s-gue170111w}

Let $W_1$ be an open set in $\Real^{N_1}$ and let $W_2$ be an open set in $\Real^{N_2}$. Let $E$ and $F$ be vector bundles over $W_1$ and $W_2$, respectively. 
An $m$-dependent continuous operator
$A_m: \mathcal{C}^\infty_c(W_2,F)\To\mathcal{D}'(W_1,E)$ is called $m$-negligible on $W_1\times W_2$
if, for $m$ large enough, $A_m$ is smoothing and, for any $K\Subset W_1\times W_2$, any
multi-indices $\alpha$, $\beta$ and any $N\in\mathbb N$, there exists $C_{K,\alpha,\beta,N}>0$
such that
\[
\abs{\pr^\alpha_x\pr^\beta_yA_m(x, y)}\leq C_{K,\alpha,\beta,N}m^{-N}\:\: \text{on $K$},\ \ \forall m\gg1.
\]
In that case we write
\[A_m(x,y)=O(m^{-\infty})\:\:\text{on $W_1\times W_2$,} \quad
\text{or} \quad
A_m=O(m^{-\infty})\:\:\text{on $W_1\times W_2$.}\]
If $A_m, B_m: \mathcal{C}^\infty_c(W_2, F)\To\mathcal{D}'(W_1, E)$ are $m$-dependent continuous operators,
we write $A_m= B_m+O(m^{-\infty})$ on $W_1\times W_2$ or $A_m(x,y)=B_m(x,y)+O(m^{-\infty})$ on $W_1\times W_2$ if $A_m-B_m=O(m^{-\infty})$ on $W_1\times W_2$. 
When $W=W_1=W_2$, we sometime write ``on $W$".

Let $X$ and $M$ be smooth manifolds and let $E$ and $F$ be vector bundles over $X$ and $M$, respectively. Let $A_m, B_m: \mathcal{C}^\infty(M,F)\To\mathcal{C}^\infty(X,E)$ be $m$-dependent smoothing operators. We write $A_m=B_m+O(m^{-\infty})$ on $X\times M$ if on every local coordinate patch $D$ of $X$ and local coordinate patch $D_1$ of $M$, $A_m=B_m+O(m^{-\infty})$ on $D\times D_1$.
When $X=M$, we sometime write on $X$.

We recall the definition of the semi-classical symbol spaces

\begin{defn} \label{d-gue140826}
Let $W$ be an open set in $\Real^N$. Let
\[
S^0_{{\rm loc\,}}(1;W):=\Big\{(a(\cdot,m))_{m\in\Real}\,;\, \forall\alpha\in\mathbb N^N_0,
\forall \chi\in\mathcal{C}^\infty_c(W)\,:\:\sup_{m\in\Real, m\geq1}\sup_{x\in W}\abs{\pr^\alpha(\chi a(x,m))}<\infty\Big\}\,.
\]
For $k\in\Real$, let
\[
S^k_{{\rm loc}}(1):=S^k_{{\rm loc}}(1;W)=\Big\{(a(\cdot,m))_{m\in\Real}\,;\,(m^{-k}a(\cdot,m))\in S^0_{{\rm loc\,}}(1;W)\Big\}\,.
\]
Hence $a(\cdot,m)\in S^k_{{\rm loc}}(1;W)$ if for every $\alpha\in\mathbb N^N_0$ and $\chi\in\mathcal{C}^\infty_c(W)$, there
exists $C_\alpha>0$ independent of $m$, such that $\abs{\pr^\alpha (\chi a(\cdot,m))}\leq C_\alpha m^{k}$ holds on $W$.

Consider a sequence $a_j\in S^{k_j}_{{\rm loc\,}}(1)$, $j\in\N_0$, where $k_j\searrow-\infty$,
and let $a\in S^{k_0}_{{\rm loc\,}}(1)$. We say
\[
a(\cdot,m)\sim
\sum\limits^\infty_{j=0}a_j(\cdot,m)\:\:\text{in $S^{k_0}_{{\rm loc\,}}(1)$},
\]
if, for every
$\ell\in\N_0$, we have $a-\sum^{\ell}_{j=0}a_j\in S^{k_{\ell+1}}_{{\rm loc\,}}(1)$ .
For a given sequence $a_j$ as above, we can always find such an asymptotic sum
$a$, which is unique up to an element in
$S^{-\infty}_{{\rm loc\,}}(1)=S^{-\infty}_{{\rm loc\,}}(1;W):=\cap _kS^k_{{\rm loc\,}}(1)$.

Similarly, we can define $S^k_{{\rm loc\,}}(1;Y,E)$ in the standard way, where $Y$ is a smooth manifold and $E$ is a vector bundle over $Y$. 
\end{defn}

\subsection{Functional calculus for Toeplitz operators on CR manifolds with circle action}\label{s-gue211214yyda}
 
We now assume that $X$ admits a CR and transversal locally free circle action $e^{i\theta}$. Let $T$ be the vector field on $X$ given by 
\[(Tu)(x):=\frac{\pr}{\pr\theta}u(e^{i\theta}\circ x)|_{\theta=0},\ \ \forall u\in\mathcal{C}^\infty(X).\]
We take $\omega_0$ so that the associated Reeb vector field $R$ is equal to $T$. Assume that the Hermitian metric $\langle\,\cdot\,|\,\cdot\,\rangle$ on $TX\otimes\mathbb{C}$ is $S^1$-invariant. For every $m\in Z$, set 
\[\Omega^{0,q}_m(X):=\set{u\in\Omega^{0,q}(X);\, (e^{i\theta})^*u=e^{im\theta}u,\ \ \forall e^{i\theta}\in S^1}.\]
 Let $L^2_{(0,q),m}(X)$ be
the completion of $\Omega^{0,q}_m(X)$ with respect to $(\,\cdot\,|\,\cdot\,)$. 
 Put 
\[({\rm Ker\,}\Box^{q}_b)_m:=({\rm Ker\,}\Box^{q}_b)\cap L^2_{(0,q),m}(X).\]
The $m$-th Szeg\H{o} projection is the orthogonal projection 
$S^{(q)}_{m}:L^2_{(0,q)}(X)\To ({\rm Ker\,}\Box^{(q)}_b)_m$
with respect to $(\,\cdot\,|\,\cdot\,)$. Let $P\in L^{0}_{{\rm cl\,}}(X,T^{*0,q}X\boxtimes(T^{*0,q}X)^*)^{S^1}$ be a self-adjoint $S^1$-invariant classical pseudodifferential operator on $X$ with scalar principal symbol. Put 
\[T^{(q)}_{P,m}:=S^{(q)}_{m}\circ P\circ S^{(q)}_m: L^2_{(0,q)}(X)\To L^2_{(0,q),m}(X).\]
Let $\tau\in\mathcal{C}^\infty_c(\mathbb R)$. 
It is straightforward to see that 
\begin{equation}\label{e-gue211214ycda}
\tau(T^{(q)}_{P,m})(x,y)=\frac{1}{2\pi}\int^{2\pi}_0\tau(T^{(q)}_{P})(x,e^{i\theta}y)e^{im\theta}\mathrm{d}\theta.
\end{equation}

For $p\in X$, let 
\[N_p:=\set{g\in S^1;\, g\circ p=p}.\]
Let 
\[X_{{\rm reg\,}}=\set{x\in X;\, e^{i\theta}x\neq x, \forall\theta\in]0,2\pi[}.\]

From Theorem~\ref{t-gue211211yyd} and Theorem~\ref{t-gue211211yyda} and \eqref{e-gue211214ycda}, we can repeat the proof of Theorem 4.4 in~\cite{gh}  and get 

\begin{thm}\label{t-gue211214yyds}
With the notations and assumptions above, recall that we let $q=n_-$. We do not assume that $n_-\neq n_+$. Let $\tau\in\mathcal{C}^\infty_c(\mathbb R)$ with $0\notin{\rm supp\,}\tau$. Let $p\in X$ and assume that 
\[N_p=\set{e^{\frac{\ell 2\pi}{k}i};\, \ell=0,1,\ldots,k-1}.\]
Let $D$ be an open local coordinate patch of $p$ with local coordinates $x=(x_1,\ldots,x_{2n+1})$. Then, as $m\gg1$, 
\begin{equation}\label{e-gue211214yydu}
\begin{split}
&\tau(T^{(q)}_{P,m})(x,y)=\sum^{k-1}_{\ell=0}e^{\frac{im2\pi\ell}{k}}e^{im\Psi(x,e^{\frac{2\pi\ell}{k}i}y)}a_P(x,e^{\frac{2\pi\ell}{k}i}y,m)+O(m^{-\infty})\ \ \mbox{on $D\times D$},\\
&\Psi\in\mathcal{C}^\infty(D\times D),\ \ \Psi(x,x)=0,\ \ d_x\Psi(x,x)=-d_y\Psi(x,x)=-\omega_0(x),\ \ \forall x\in D,\\
&{\rm Im\,}\Psi(x,y)\geq c\inf_{e^{i\theta}\in S^1}d(e^{i\theta}x,y)^2,\ \ \forall (x,y)\in D\times D,\ \ \mbox{$c>0$ is a constant},\\
&a_P(x,y,m)\in S^{n}_{{\rm loc\,}}(1; D\times D, T^{*0,q}X\boxtimes(T^{*0,q}X)^*),\\
&\mbox{$a_P(x,y,m)\sim\sum^\infty_{j=0}m^{n-j}a_{P,j}(x,y)$ in $S^{n}_{{\rm loc\,}}(1; D\times D, T^{*0,q}X\boxtimes(T^{*0,q}X)^*)$},\\
&a_{P,j}(x,y)\in\mathcal{C}^\infty(D\times D, T^{*0,q}X\boxtimes(T^{*0,q}X)^*),\ \ j=0,1,2,\ldots,
\end{split}
\end{equation}
and 
\begin{equation}\label{e-gue211214yydv}
\mbox{$a_{P,0}(x,x)=\frac{1}{2}\pi^{-n-1}\abs{{\rm det\,}\mathcal L_x}\tau(\sigma^0_P(x,-\omega_0(x)))\pi_{x,n_-}$, for every $x\in D$}. 
\end{equation}
\end{thm} 

From Theorem~\ref{t-gue211214yyds}, we deduce 

\begin{cor}\label{c-gue211214yyds}
With the notations and assumptions above, recall that we let $q=n_-$. We do not assume that $n_-\neq n_+$. Let $\tau\in\mathcal{C}^\infty_c(\mathbb R)$ with $0\notin{\rm supp\,}\tau$. Then, there is a constant $C>0$ such that for $m\gg1$, 
\begin{equation}\label{e-gue211214yydw}
\abs{\tau(T^{(q)}_{P,m})(x,y)}\leq Cm^n,\ \ \mbox{for all $(x,y)\in X\times X$}. 
\end{equation}

Assume that $X$ is connected and $X_{{\rm reg\,}}\neq\emptyset$. Then, $X_{{\rm reg\,}}$ is dense in $X$. We have 
\begin{equation}\label{e-gue211214yydx}
\lim_{m\To+\infty}m^{-n}\tau(T^{(q)}_{P,m})(x,x)=\frac{1}{2}\pi^{-n-1}\abs{{\rm det\,}\mathcal L_x}\tau(\sigma^0_P(x,-\omega_0(x)))\pi_{x,n_-},\ \ \mbox{for every $x\in X$}.
\end{equation}
\end{cor}

\subsection{Semi-classical spectral dimensions law for Toeplitz operators on CR manifolds with circle action}\label{s-gue211216yyd}

Let $I\subset\mathbb R$ be an open bounded interval with $0\notin\overline{I}$. We are going to study the asymptotic limit of $\lim_{m\To+\infty}m^{-n}{\rm dim\,}E_I(T^{(q)}_{P,m})$. 
For every $x\in X$, let $f_1(x),\ldots,f_r(x)\in T^{*0,q}_xX$ be an orthonormal basis for $T^{*0,q}_xX$, $r={\rm dim\,}T^{*0,q}_xX$. Let $A: T^{*0,q}_yX\To T^{*0,q}_xX$ be a linear transformation, $x, y\in X$. 
We write 
\begin{equation}\label{e-gue211215yyd} A=\sum^r_{i,j=1} c_{i,j}f_i(x)\otimes f_k(y)\,,\quad c_{i,j}\in\mathbb{R}\,,\quad \forall i, j,  \end{equation}
where $f_i(x)\otimes f_k(y) : T^{*0,q}_yX\To T^{*0,q}_xX$ is the linear transformation given by $$(f_i(x)\otimes f_k(y))(f_s(y))=\delta_{k,s}f_i(x)\,,$$ for every $i, k, s=1,\ldots,r$. 
When $x=y$, we put 
\[{\rm Tr\,}A:=\sum^r_{j=1}\langle\,Af_j(x)\,|\,f_j(x)\,\rangle.\]
We need 

\begin{lem}\label{l-gue211215yyd}
With the notations and assumptions above, let $I\subset\mathbb R$ be an open bounded interval with $0\notin\overline{I}$. Let $\tau\in\mathcal{C}^\infty_c(\mathbb R)$ with $0\notin{\rm supp\,}\tau$ and $\tau\equiv1$ on some neighborhood of $I$. We have 
\begin{equation}\label{e-gue211215yydI}
{\rm Tr\,}(\Pi_I(T^{(q)}_{P,m})(x,x))\leq{\rm Tr\,}(\tau(T^{(q)}_{P,m})^2(x,x)),\ \ \mbox{for every $x\in X$}. 
\end{equation}
\end{lem}

\begin{proof}
Since $\tau\equiv1$ on some neighborhood of $I$, we have 
\begin{equation}\label{e-gue211215yydII}
\Pi_I(T^{(q)}_{P,m})=\tau(T^{(q)}_{P,m})\Pi_I(T^{(q)}_{P,m}).
\end{equation}
From \eqref{e-gue211215yydII}, we have 
\begin{equation}\label{e-gue211215yydIII}
\Pi_I(T^{(q)}_{P,m})(x,x)=\int_X\tau(T^{(q)}_{P,m})(x,y)\Pi_I(T^{(q)}_{P,m})(y,x)\mathrm{dV}_X(y).
\end{equation}
For every $x\in X$, let $f_1(x),\ldots,f_r(x)\in T^{*0,q}_xX$ be an orthonormal basis for $T^{*0,q}_xX$, $r={\rm dim\,}T^{*0,q}_xX$.
 As \eqref{e-gue211215yyd}, we write 
\begin{equation}\label{e-gue211215yyda}
\begin{split}
\Pi_I(T^{(q)}_{P,m})(x,y)=\sum^r_{i,j=1} a_{i,j}(x,y)f_i(x)\otimes f_k(y),\ \ a_{i,j}(x,y)\in\mathbb R,\ \ \forall i, j, \\
\tau(T^{(q)}_{P,m})(x,y)=\sum^r_{i,j=1} b_{i,j}(x,y)f_i(x)\otimes f_k(y),\ \ b_{i,j}(x,y)\in\mathbb R,\ \ \forall i, j.
\end{split}
\end{equation}
From \eqref{e-gue211215yydIII}, we have 
\begin{equation}\label{e-gue211215yydb}
\begin{split}
&{\rm Tr\,}(\Pi_I(T^{(q)}_{P,m})(x,x))=\sum^r_{i=1}a_{i,i}(x,x)=\sum^r_{i,j=1}\int b_{i,j}(x,y)a_{j,i}(y,x)\mathrm{dV}_X(y)\\
&\leq\sum^r_{i,j=1}\sqrt{\int\abs{b_{i,j}(x,y)}^2\mathrm{dV}_X(y)\int\abs{a_{j,i}(y,x)}^2\mathrm{dV}_X(y)}\\
&\leq\sqrt{\sum^r_{i,j=1}\int\abs{b_{i,j}(x,y)}^2\mathrm{dV}_X(y)}\sqrt{\int\abs{a_{j,i}(y,x)}^2\mathrm{dV}_X(y)}.
\end{split}
\end{equation}
Write 
\begin{equation}\label{e-gue211215yyde}
\tau(T^{(q)}_{P,m})^2(x,y)=\sum^r_{i,j=1} d_{i,j}(x,y)f_i(x)\otimes f_k(y),\ \ d_{i,j}(x,y)\in\mathbb R,\ \ \forall i, j.
\end{equation}
It is not difficult to see that 
\begin{equation}\label{e-gue211215yydc}
\begin{split}
&\sum^r_{i,j=1}\int\abs{b_{i,j}(x,y)}^2\mathrm{dV}_X(y)=\sum^r_{i=1}d_{i,i}(x,x)={\rm Tr\,}(\tau(T^{(q)}_{P,m})^2(x,x)),\\
&\sum^r_{i,j=1}\int\abs{a_{i,j}(x,y)}^2\mathrm{dV}_X(y)=\sum^r_{i=1}a_{i,i}(x)={\rm Tr\,}(\Pi_I(T^{(q)}_{P,m})(x,x)).
\end{split}
\end{equation}
From \eqref{e-gue211215yydb} and \eqref{e-gue211215yydc}, we get \eqref{e-gue211215yydI}. 
\end{proof}

We can repeat the proof of Lemma~\ref{l-gue211215yyd} with minor change and deduce 

\begin{lem}\label{l-gue211215yydI}
With the notations and assumptions above, let $I\subset\mathbb R$ be an open bounded interval with $0\notin\overline{I}$. Let $\tau\in\mathcal{C}^\infty_c(I)$ 
with $0\notin{\rm supp\,}\tau$. We have 
\begin{equation}\label{e-gue211215yydp}
{\rm Tr\,}(\tau(T^{(q)}_{P,m})(x,x))\leq\sqrt{{\rm Tr\,}(\Pi_I(T^{(q)}_{P,m})(x,x))}\sqrt{{\rm Tr\,}(\tau(T^{(q)}_{P,m})^2(x,x))},\ \ \mbox{for every $x\in X$}. 
\end{equation}
Hence, 
\begin{equation}\label{e-gue211215yydq}
2{\rm Tr\,}(\tau(T^{(q)}_{P,m})(x,x))-{\rm Tr\,}(\tau(T^{(q)}_{P,m})^2(x,x))\leq{\rm Tr\,}(\Pi_I(T^{(q)}_{P,m})(x,x)),\ \ \mbox{for every $x\in X$}. 
\end{equation}
\end{lem}

We can now prove the following semi-classical spectral dimensions for Toeplitz operators

\begin{thm}\label{t-gue211215yyd}
With the notations and assumptions above, let $I\subset\mathbb R$ be an open bounded interval with $0\notin\overline{I}$. Suppose that $X$ is connected and $X_{{\rm reg\,}}\neq\emptyset$.
We have 
\begin{equation}\label{e-gue211215yydr}
\lim_{m\To+\infty}m^{-n}{\rm dim\,}E_I(T^{(q)}_{P,m})=\frac{1}{2}\pi^{-n-1}\int_{\set{x\in X;\, \sigma^0_P(x,-\omega_0(x))\in I}}\abs{{\rm det\,}\mathcal L_x}\mathrm{dV}_X(x). 
\end{equation}
\end{thm}

\begin{proof}
 Let $\set{\tau_j}^{+\infty}_{j=1}\subset\mathcal{C}^\infty_c(\mathbb R)$ with $0\notin{\rm supp\,}\tau_j$, $\tau_j\equiv1$ on some neighborhood of $I$, for every $j$ and 
 $\lim_{j\To+\infty}\tau_j=1_I$ pointwise, where $1_I(x)=1$ if $x\in I$, $1_I(x)=0$ if $x\notin I$.  Let $\set{\hat\tau_j}^{+\infty}_{j=1}\subset\mathcal{C}^\infty_c(I)$ with $0\notin{\rm supp\,}\hat\tau_j$, for every $j$ and 
 $\lim_{j\To+\infty}\hat\tau_j=1_I$ pointwise. From \eqref{e-gue211215yydI} and \eqref{e-gue211215yydq}, we have for every $j=1,2,\ldots$, 
 \begin{equation}\label{e-gue211215yyds}
 \begin{split}
& m^{-n}\int_X\Bigr(2{\rm Tr\,}(\hat\tau_j(T^{(q)}_{P,m})(x,x))-{\rm Tr\,}(\hat\tau_j(T^{(q)}_{P,m})^2(x,x))\Bigr)\mathrm{dV}_X(x)\\
 &\leq m^{-n}\int_X{\rm Tr\,}(\Pi_I(T^{(q)}_{P,m})(x,x))\mathrm{dV}_X(x)\\
 &\leq m^{-n}\int_X{\rm Tr\,}(\tau_j(T^{(q)}_{P,m})^2(x,x))\mathrm{dV}_X(x). 
 \end{split}
 \end{equation}
 Note that 
 \[{\rm dim\,}E_I(T^{(q)}_{P,m})=\int_X{\rm Tr\,}(\Pi_I(T^{(q)}_{P,m})(x,x))\mathrm{dV}_X(x).\]
 From this observation, \eqref{e-gue211214yydx}, \eqref{e-gue211215yyds} and take $j\To+\infty$ in \eqref{e-gue211215yyds}, the theorem follows. 
\end{proof}

\subsection{Quantization commutes with reduction on CR manifolds with circle action}\label{s-gue211216yydI} 

Now, we assume that $X$ admits a compact Lie group action $G$ of dimension $d$. We assume that Assumption~\ref{a-gue170123I} and Assumption~\ref{a-gue170123II} hold. We assume further that 

\begin{ass}\label{a-gue170128} The Reeb vector field $R$ is transversal to the space $\underline{\mathfrak{g}}$ at every point $p\in\mu^{-1}(0)$. Furthermore, we assume that the circle action commutes with the one of $G$ on $X$:
	\begin{equation}\label{e-gue170111ryII}
		e^{i\theta}\circ g\circ x=g\circ e^{i\theta}\circ x,\  \ 
		\mbox{for every $x\in X$, $\theta\in[0,2\pi[$, $g\in G$}\,, 
	\end{equation}
Eventually, we assume
	\[
	\mbox{$G\times S^1$ acts locally free near $\mu^{-1}(0)$}\,. 
	\]
\end{ass} 

Fix a $S^1\times G$-invariant smooth Hermitian metric $\langle\, \cdot \,|\, \cdot \,\rangle$ on $TX\otimes \mathbb{C}$ so that $T^{1,0}X$ is orthogonal to $T^{0,1}X$, $\underline{\mathfrak{g}}$ is orthogonal to $HY\cap JHY$ at every point of $Y$, $\langle \, u \,|\, v \, \rangle$ is real if $u, v$ are real tangent vectors and $\langle\,R\,|\,R\,\rangle=1$. For every $m\in \mathbb{Z}$, we set 
\[\Omega^{0,q}_m(X)^G:=\set{u\in\Omega^{0,q}(X)^G;\, (e^{i\theta})^*u=e^{im\theta}u,\ \ \forall e^{i\theta}\in S^1}.\]
 Let $L^2_{(0,q),m}(X)^G$ be
the completion of $\Omega^{0,q}_m(X)^G$ with respect to $(\,\cdot\,|\,\cdot\,)$. 
 Put 
$({\rm Ker\,}\Box^{q}_b)^G_m:=({\rm Ker\,}\Box^{q}_b)^G\cap L^2_{(0,q),m}(X).$
Let 
\[H^q_{b,m}(X)^G:=({\rm Ker\,}\Box^{q}_b)^G_m.\] 
 From Corollary~\ref{c-gue211213yydam}, we get

\begin{thm}\label{t-gue211216yyd}
With the same notations and assumptions above, if $n_-=n_+$, $n_{X_G,-}=n_{X_G,+}$ or $n_-\neq n_+$, $n_{X_G,-}\neq n_{X_G,+}$, then
\[H^{n_-}_{b,m}(X)^G\cong H^{n_{X_G,-}}_{b,m}(X_G)\]
for $\abs{m}\gg1$. 

If $n_-=n_+$ and $n_{X_G,-}\neq n_{X_G,+}$, then 
\[H^{n_-}_{b,m}(X)^G\cong H^{n_{X_G,-}}_{b,m}(X_G)\oplus H^{n_{X_G,+}}_{b,m}(X_G)\]
for $\abs{m}\gg1$. 

If $n_-\neq n_+$ and $n_{X_G,-}=n_{X_G,+}$, then 
\[H^{n_-}_{b,m}(X)^G\oplus H^{n_+}_{b,m}(X)^G\cong H^{n_{X_G,-}}_{b,m}(X_G)\]
for $\abs{m}\gg1$. 
\end{thm}

\subsection{Semi-classical spectral dimensions for invariant Toeplitz operators on CR manifolds with circle action}\label{s-gue211216yydII} 

The $m$-th $G$-invariant Szeg\H{o} projection is the orthogonal projection 
$$S^{(q)}_{G,m}:L^2_{(0,q)}(X)\To ({\rm Ker\,}\Box^{(q)}_b)^G_m$$
with respect to $(\,\cdot\,|\,\cdot\,)$.
Let $P\in L^{0}_{{\rm cl\,}}(X,T^{*0,q}X\boxtimes(T^{*0,q}X)^*)^{G\times S^1}$ be a self-adjoint $G\times S^1$-invariant pseudodifferential operator on $X$ with scalar principal symbol. Put 
\[T^{(q)}_{P,G,m}:=S^{(q)}_{G,m}\circ P\circ S^{(q)}_{G,m}: L^2_{(0,q)}(X)\To L^2_{(0,q),m}(X)^G.\] 

We can now prove the following semi-classical spectral dimensions for $G$-invariant Toeplitz operators

\begin{thm}\label{t-gue211215yydm}
With the notations and assumptions above, let $I\subset\mathbb R$ be an open bounded interval with $0\notin\overline{I}$. Suppose that $X$ is connected and $X_{{\rm reg\,}}\neq\emptyset$.. 
We have 
\begin{equation}\label{e-gue211215yydrm}
\lim_{m\To+\infty}m^{-n+d}{\rm dim\,}E_I(T^{(q)}_{P,G,m})=\frac{1}{2}\pi^{-n+d-1}\int_{\set{x\in X_G;\, \sigma^0_{P_{X_G}}(x,-\omega_0(x))\in I}}\abs{{\rm det\,}\mathcal L_{X_G,x}}\mathrm{dV}_{X_G}(x). 
\end{equation}
\end{thm} 

\begin{proof}
Let $I_1$ be an open bounded interval with $0\notin\overline{I_1}$, $I\subset I_1$, $I\neq I_1$. Let $\tau\in\mathcal{C}^\infty_c(I_1)$ with $0\notin{\rm supp\,}\tau$ and $\tau\equiv1$ on some neighborhood of $I$. From 
Theorem~\ref{t-gue211214yydI} and notice that ${\rm dim\,}E_\tau(T^{(n_+)}_{P,G,m})={\rm dim\,}E_\tau(T^{(n_{X_G,+})}_{P,G,m})=0$ if $m\gg1$, we conclude that if $m\gg1$, we have 
\begin{equation}\label{e-gue211216yyda}
{\rm dim\,}E_I(T^{(q)}_{P,G,m})\leq{\rm dim\,}E_{\tau}(T^{(q)}_{P,G,m})\leq{\rm dim\,}E_{I_1}(T^{(n_{X_G,-})}_{P,G,m})\,.
\end{equation}
From \eqref{e-gue211215yydr} and \eqref{e-gue211216yyda}, we get 
\begin{equation}\label{e-gue211216yydb}
\begin{split}
&\lim_{m\To+\infty}m^{-n+d}{\rm dim\,}E_I(T^{(q)}_{P,G,m})\\
&\leq\lim_{m\To+\infty}m^{-n+d}{\rm dim\,}E_{I_1}(T^{(n_{X_G,-})}_{P,G,m})=
\frac{1}{2}\pi^{-n+d-1}\int_{\widetilde{X_G}}\abs{{\rm det\,}\mathcal L_{X_G,x}}\mathrm{dV}_{X_G}(x)
\end{split}
\end{equation}
where
\[\widetilde{X_G}:=\set{x\in X_G;\, \sigma^0_{P_{X_G}}(x,-\omega_0(x))\in I_1}\,.\]

Let $\hat I_1$ be an open interval with $0\notin\overline{\hat I_1}$, $\hat I_1\subset I$, $\hat I_1\neq I$. Let $\hat\tau\in\mathcal{C}^\infty_c(I)$ with $0\notin{\rm supp\,}\hat\tau$ and $\hat\tau\equiv1$ on some neighborhood of $\hat I_1$. From Theorem~\ref{t-gue211214yydI} again, if $m\gg1$, we have 
\begin{equation}\label{e-gue211216yydc}
{\rm dim\,}E_I(T^{(q)}_{P,G,m})\geq {\rm dim\,}E_{\hat\tau}(T^{(q)}_{P,G,m})\geq{\rm dim\,}E_{\hat I_1}(T^{(n_{X_G,-})}_{P,G,m}).
\end{equation}
From \eqref{e-gue211215yydr} and \eqref{e-gue211216yydc}, we get 
\begin{equation}\label{e-gue211216yyde}
\begin{split}
&\lim_{m\To+\infty}m^{-n+d}{\rm dim\,}E_I(T^{(q)}_{P,G,m})\\
&\geq\lim_{m\To+\infty}m^{-n+d}{\rm dim\,}E_{\hat I_1}(T^{(n_{X_G,-})}_{P,G,m})=
\frac{1}{2}\pi^{-n+d-1}\int_{\widehat{X_G}}\abs{{\rm det\,}\mathcal L_{X_G,x}}\mathrm{dV}_{X_G}(x)
\end{split}
\end{equation}
where
\[\widehat{X_G}:= \set{x\in X_G;\, \sigma^0_{P_{X_G}}(x,-\omega_0(x))\in\hat I_1} \]
From \eqref{e-gue211216yydb} and \eqref{e-gue211216yyde}, we get \eqref{e-gue211215yydrm}. 
\end{proof} 

\textbf{Acknowledgements:} This project was started during the first author’s postdoctoral fellowship at the National Center for Theoretical Sciences in Taiwan; we thank the Center for the support. Chin-Yu Hsiao was partially supported by Taiwan Ministry of Science and Technology projects  108-2115-M-001-012-MY5, 109-2923-M-001-010-MY4.

\end{document}